\documentclass{article}

\RequirePackage[amsmath, amsthm, noinfoline]{imsart}

\usepackage{amssymb, graphicx, accents, fullpage, caption,color}

\usepackage{tabularx}

\usepackage[labelformat=simple]{subcaption}

\usepackage{float}

\def\Uh{\hat{U}}

\def\Ft{\mathcal{F}_{t}}
\def\E{\mathcal{E}}
\def\H{\mathcal{H}}
\def\pih{\hat{\pi}}

\newcommand{\U}[1]{U^{(#1)}}

\def\Umer{U^{\text{Mer}}}

\def\ul{\underline}

\newcommand{\ol}[1]{\overline{#1}}
\newcommand{\C}[1]{c_{#1}}

\newtheorem{lemma}{Lemma}[section]
\newtheorem{theorem}{Theorem}[section]

\theoremstyle{remark}
\newtheorem{remark}{Remark}[section]

\theoremstyle{assumption}
\newtheorem{assumption}{Assumption}[section]

\newtheoremstyle{case}
{\topsep}
{\topsep}
{\itshape}
{30pt} 
{\it} 
{}
{ }
{\thmname{#1}\thmnumber{ #2}:\thmnote{ (#3)}}

\theoremstyle{case}
\newtheorem{case}{Case}

\theoremstyle{definition}
\newtheorem{definition}{Definition}[section]

\DeclareMathOperator*{\esssup}{ess\,sup}

\numberwithin{equation}{section}

\begin{document}
\begin{frontmatter}
\title{Asymptotic approximation of optimal portfolio for small time horizons}

\author{\fnms{Rohini Kumar}\thanks{{\tt rkumar@math.wayne.edu}}
}
\address{\noindent Dept. Mathematics, \\
 Wayne State University, \\
 Detroit, MI 48202
}

\author{\fnms{Hussein Nasralah}\thanks{{\tt hussein.nasralah@wayne.edu}}
}
\address{\noindent Dept. Mathematics, \\
 Wayne State University, \\
 Detroit, MI 48202\\
 \quad \\
  \quad \\
\today}

\smallskip
\smallskip

\begin{abstract}
We consider the problem of portfolio optimization in a simple incomplete market and under a general utility function.  By working with the associated Hamilton-Jacobi-Bellman partial differential equation (HJB PDE), we obtain a closed-form formula for a trading strategy which approximates the optimal trading strategy when the time horizon is small.  This strategy is generated by a first order approximation to the value function.  The approximate value function is obtained by constructing classical  sub- and super-solutions to the HJB PDE using a formal expansion in powers of  horizon time. Martingale inequalities are used to sandwich the true value function between the constructed sub- and super-solutions.  A rigorous proof of the accuracy of the approximation formulas is given. We end with a heuristic scheme for extending our small-time approximating formulas to approximating formulas in a finite time horizon.

\end{abstract}

\end{frontmatter}

\section{Introduction}
In this paper, we study the problem of portfolio optimization when the time horizon is small. For simplicity, we consider a financial market with two assets, one risky and one risk-free.  Given a pair $(t,x) \in [0,T]\times(0,\infty)$ of initial time $t$ and initial wealth $x$, an investor wishes to invest in such a way as to maximize her expected utility of wealth at time $T$ given today's information.  Specifically, if $U_{T}(x)$ is a function modeling the investor's utility of wealth at the terminal time $T$, then the investor wishes to choose a portfolio $\pi$ which maximizes $E[U_{T} | \Ft]$ ($\Ft$ being the sigma-algebra that informally represents information up to time $t$).

Under Markovian assumptions on the price process of the risky asset, this optimization problem can be studied via the associated HJB equation, as, for example, in \cite{Lorig2016LSV,Merton69Lifetime,Nadtochiy13ApproxOIP,Zariphopoulou01Val}.  This portfolio optimization problem was first studied in a continuous time setting by Merton \cite{Merton69Lifetime, Merton71Optimum} in a complete market.  The utility maximization problem can also be studied using duality arguments as in  \cite{CvitanicKaratzas92ConvexDuality,KaratzasEtAl91MartingaleAndDuality,KramkovSchachermayer99AsymptoticElasticity,Schachermayer01NegativeWealth}.

Portfolio optimization has also been studied under the assumption of an infinite time horizon, for example in \cite{Merton69Lifetime, Merton71Optimum, Pang04InfiniteTimePower, Pang06InfiniteTimeLog}.   In \cite{Pang04InfiniteTimePower, Pang06InfiniteTimeLog}, the author studies the problem of optimal investment and consumption in an infinite time horizon assuming a model with stochastic interest rate and where risky asset price is a geometric Brownian motion.

Tehranchi \cite{Tehranchi04ExpInc} studied the problem under the assumption of an incomplete market, where  the market is driven by two Brownian motions and the asset price is not necessarily Markovian. The absence of the Markovian structure precludes the use of the dynamic programming principle.  By proving H\"older-type inequalities for functionals of correlated Brownian motions, Tehranchi \cite{Tehranchi04ExpInc} was able to study the portfolio optimization problem when the utility function is a product of a function of the wealth and a function of a correlated stochastic factor.   Explicit formulas were obtained in \cite{Tehranchi04ExpInc} when the function of wealth is an exponential function, a logarithmic function, and a power function.

In an incomplete market, few explicit formulas for the optimal portfolio exist in the literature and attempts have been made to obtain approximating formulas.  We mention some of the work in this direction where the risky asset price model is Markovian and has correlated stochastic factors.   In \cite{Zariphopoulou01Val}, the utility function is assumed to be of Constant Relative Risk Aversion (CRRA) type, i.e., a product of a power function in wealth and a function depending on the stochastic factor. Under this assumption, Zariphopoulou in \cite{Zariphopoulou01Val} is able to obtain the value function in terms of the solution to a linear parabolic PDE. The results in \cite{Zariphopoulou01Val} have proved useful for computing explicit formulas for specific examples.   
In \cite{Lorig2016LSV}, Lorig and Sircar consider the problem of portfolio optimization in finite horizon assuming a local stochastic volatility model for a risky asset. They use a Taylor series expansion of the model coefficients to obtain approximating formulas for the value function and optimal portfolio. While approximating formulas are obtained for general utility functions, accuracy of the approximation is established only in the case of power utilities.    Fouque et al., in  \cite{Fouque2015Portfolio}, assume a model with multiscale stochastic factors and by asymptotic analysis  obtain approximating formulas for the optimal portfolio.

In \cite{Fouque2015Portfolio,Lorig2016LSV}, well-posedness of the associated HJB equation is not established and the authors work under the assumption that the value function is the classical solution of the HJB equation with a sufficient degree of regularity.  In \cite{Nadtochiy13ApproxOIP}, as well as in our paper, no such assumption is made. In \cite{Nadtochiy13ApproxOIP}, Nadtochiy and Zariphopoulou state that their model is the ``simplest and most direct extension" \cite{Nadtochiy13ApproxOIP} to an incomplete market of the model introduced by Merton in \cite{Merton69Lifetime,Merton71Optimum}. In this model, the utility function depends only on wealth.  Instead of working directly with the associated HJB equation, the authors work with the marginal HJB equation, which they prove has a unique viscosity solution (see \cite{Crandall92UGvisc} for more information on viscosity solutions).  Without assuming the value function satisfies the HJB equation,  Nadtochiy and Zariphopoulou, in  \cite{Nadtochiy13ApproxOIP}, prove that the integral of the viscosity solution of the marginal HJB equation is indeed the value function.  In addition, the authors derive approximations to the optimal portfolio which they term ``$\epsilon$-optimal portfolios" \cite{Nadtochiy13ApproxOIP}.

In this paper, we find a closed-form formula for an approximation to the optimal portfolio in a small time horizon under a stochastic volatility model for the risky asset price.  As the well-posedness of the associated HJB equation is not established,  we do not assume the value function is a classical solution of the HJB equation.  Additionally, we do not assume a specific form for the utility function. We only assume that the asymptotic behavior of our utility function as wealth approaches $0$ or $\infty$ is as a logarithmic utility, or sum of power utilities.  Accuracy of the approximation is established. We then use the small time approximation to iteratively build an approximation on longer time horizons.

Our discussion and results are organized as follows: we state our model assumptions, as well as our assumptions on the behavior of the utility function and its derivatives, in section 2.  The main theorem is proved in section 3.  In section 4, we build our close-to-optimal portfolio and verify the degree of closeness. In section 5, we graphically illustrate our small time approximation by an example.  The results in section 4 are then extended to longer time horizons in section 6 and applied to an example.

\section{Model and Assumptions}

We consider the following simple incomplete market model, as in \cite{Nadtochiy13ApproxOIP}, however our assumptions on the terminal utility function will be different from those considered in \cite{Nadtochiy13ApproxOIP}. 

Consider a market consisting of one risky asset (e.g., a stock) with price process $S_{t}$ and one riskless asset (e.g., a bond).  The price process of the risky asset satisfies
	\begin{equation} \label{stockprice}
	dS_{t} = \mu(Y_{t}) S_{t}\,dt + \sigma(Y_{t}) S_{t} \,dW^{1}_{t},
	\end{equation}
where $Y_{t}$ is a stochastic factor which evolves as
	\begin{equation} \label{stochasticfactor}
	dY_{t} = b(Y_{t}) \,dt + a(Y_{t}) (\rho dW^{1}_{t} + \sqrt{1 - \rho^{2}} \,dW^{2}_{t}).
	\end{equation}
The vector $W_{t} = (W^{1}_{t}, W^{2}_{t})$ is a two-dimensional standard Brownian motion adapted to the natural filtration $(\mathcal{F}_{t})_{t \in [0,T]}$ given by $\mathcal{F}_{t} = \sigma(W_{s} : 0 \leq s \leq t )$, and $\rho$ satisfies $-1 < \rho < 1$.  We also define the Sharpe ratio $\lambda(Y_{t}) := \frac{\mu(Y_{t}) - r}{\sigma(Y_{t})}$, where $r$ is the risk free interest rate.

\begin{assumption}[Model Assumptions]\label{modelassumptions}
Denote as $C(\Bbb R)$ the space of continuous functions $f: \Bbb R \to \Bbb R$, while $C^{k}(\Bbb R)$ is the space of $k$-times continuously differentiable functions $g: \Bbb R \to \Bbb R$ (for $k \geq 1, k \in \Bbb N$).  The coefficients in the stochastic differential equations (SDEs) \eqref{stockprice} and \eqref{stochasticfactor}, as well as $\lambda$, satisfy the following conditions (as in Assumption 1 of \cite{Nadtochiy13ApproxOIP}):
	\begin{enumerate}
	\item $\mu, \sigma \in C(\Bbb R)$ with $\sigma$ strictly positive.
	\item $b \in C^{1}(\Bbb R)$ and $\lambda, a \in C^{2}(\Bbb R)$ with $a$ strictly positive.
	\item There exists a constant $\C1>0$ such that 
		\begin{equation*}
		|a| + |\frac{1}{a}| + |a'| + |a''| + |b| + |b'| + |\lambda| + |\lambda'| + |\lambda''| \leq \C1.
		\end{equation*}
	\end{enumerate}
\end{assumption}	

\begin{assumption}[Utility Assumptions] \label{utilityassumptions} 
	We denote the investor's terminal utility function by $U_{T}(x)$.  We assume $U_{T}(x)$ is a strictly increasing, concave function belonging to $C^{5}(\Bbb R)$.  In addition, we make the following assumptions on the asymptotic growth of the utility function: \\$U_T(x)$ is such that, either
	
	\begin{case} Conditions \eqref{u(x)} - \eqref{uprime2(x)} hold for $M(x) := \log(x)$; \end{case} or
	\begin{case} Conditions \eqref{u(x)} - \eqref{uprime2(x)} hold for $M(x) :=\frac{x^{1-\alpha}}{1-\alpha} +\frac{x^{1-\beta}}{1-\beta}$, $\alpha, \beta \neq 1$ and positive. \end{case} 

Asymptotic growth conditions:
		\begin{equation} \label{u(x)}
	0 < \inf \limits_{x > 0} \left (\frac{U_{T}'(x)}{M'(x)}\right ) \leq \sup \limits_{x > 0} \left (\frac{U_{T}'(x)}{M'(x)}\right )  < \infty \\
		\end{equation} 
	
		\begin{equation} \label{uprime(x)}
		0 < \inf \limits_{x > 0} \left (\frac{U_{T}''(x)}{M''(x)}\right ) \leq \sup \limits_{x > 0} \left (\frac{U_{T}''(x)}{M''(x)}\right )  < \infty \\
		\end{equation}

		\begin{equation}\label{Rprime(x)}
		0 < \inf \limits_{x > 0} \left (\frac{U_{T}^{(3)}(x)}{M^{(3)}(x)}\right ) \leq \sup \limits_{x > 0} \left (\frac{U_{T}^{(3)}(x)}{M^{(3)}(x)}\right )  < \infty \\
		\end{equation}		
			
		\begin{equation}\label{uprime2(x)}
		0 < \inf \limits_{x > 0} \left (\frac{U_{T}^{(4)}(x)}{M^{(4)}(x)}\right ) \leq \sup \limits_{x > 0} \left (\frac{U_{T}^{(4)}(x)}{M^{(4)}(x)}\right )  < \infty. \\
		\end{equation}	
\end{assumption}

	\begin{remark}
	These assumptions allow for any strictly increasing, concave utility function in $C^5(\mathbb R)$ that behaves as a logarithmic function, or as different power functions, asymptotically as wealth approaches $0$ and $\infty$.  Three examples of such utility functions are:
		\begin{enumerate}
		\item {\it Power utility:} $U_{T}(x) = \frac{x^{1-\gamma}}{1-\gamma}$, for some positive $\gamma \neq 1$. 
		\item {\it Mixture of power utilities: $U_{T}(x) = \C1\frac{x^{1-\alpha}}{1-\alpha} + \C2\frac{x^{1-\beta}}{1-\beta}$, for $\C1, \C2 > 0$ and for positive $\alpha, \beta \neq 1$}.
		\item {\it Log utility:} $U_{T}(x) = \log(x)$.
		\end{enumerate}
	\end{remark}
	\begin{remark}
	While the utility assumptions in Assumption \ref{utilityassumptions} are similar to those in \cite{Nadtochiy13ApproxOIP}, these assumptions allow for logarithmic utility functions, as well as utility functions described by mixtures of power functions; these two examples are not covered by the results in paper \cite{Nadtochiy13ApproxOIP}.
	\end{remark}
	As the investor will be investing in one risky and one risk-free asset, $\pi_{t}$ will denote the discounted amount of wealth invested into the risky asset at time $t$.  We wish to consider only self-financing trading strategies, and thus, denoting by $\pi^{0}_{t}$ the discounted amount of wealth invested in the risk-free asset, choosing $\pi_{t}$ necessarily implies the value of $\pi^{0}_{t}$.  Because of this, as in \cite{Nadtochiy13ApproxOIP}, we will identify each trading strategy with the amount $\pi_{t}$ invested in the risky asset.  From this we can define the discounted wealth process $X^{\pi}_{t} := \pi_{t} + \pi^{0}_{t}$.  This process evolves in the following way
	\begin{equation} \label{dscw}
	dX^{\pi}_{t} = \sigma(Y_{t})\pi_{t}(\lambda(Y_{t})\,dt + \,dW^{1}_{t}) \text{ for } t \in [0,T].
	\end{equation}
Henceforth, we will denote $\sigma(Y_{s})$ by $\sigma_{s}$ and $\lambda(Y_{s})$ by $\lambda_{s}$.
	
\begin{definition}[Admissible Trading Strategies] \label{admissibledefinition}
The only strategies considered will be {\it admissible} strategies, meaning: 
	\begin{enumerate}
	\item $\pi_{t}$ is progressively measurable with respect to the natural filtration of our two-dimensional Brownian motion.
	\item $E  \displaystyle \int \limits_{0}^{T} \sigma_{s}^{2} \pi_{s}^{2} \,ds   < \infty$.
	\item Given an initial wealth $x \in (0,\infty)$, the discounted wealth process \eqref{dscw}  is strictly positive for all $t \in [0,T]$.  
	\item $E \left [\int \limits_{0}^{T} (X^{\pi}_{s})^{-2\gamma} \sigma_{s}^{2}\pi_{s}^{2}  \,ds \right ]  < \infty$, where $\gamma = 1$ under Case 1 of Assumption \ref{utilityassumptions}, and $\gamma := \max\{\alpha, \beta\} > 1$ under Case 2 of Assumption \ref{utilityassumptions}.
	\end{enumerate}
\end{definition}

Denoting the set of admissible trading strategies as $\mathcal{A}$, we define the value function, $J(t,x,y)$, as 
	\begin{equation} \label{valuefunction}
	J(t,x,y) := \esssup \limits_{\pi \in \mathcal{A}} E[ U_{T}(X^{\pi}_{T}) | X^{\pi}_{t} = x, Y_{t} = y ].
	\end{equation}
The value function $J$ is formally a solution to the Hamilton-Jacobi-Bellman (HJB) equation given by 
	\begin{equation} \label{hjb}
	\begin{cases}
	\begin{split}
	U_{t} &+ \max \limits_{\pi} \left (  \frac{1}{2} \sigma^{2}(y)\pi^{2}U_{xx} + \pi(\sigma(y) \lambda(y) U_{x} + \rho \sigma(y) a(y) U_{xy})  \right )\\& + \frac{1}{2}a^{2}(y)U_{yy} + b(y)U_{y} = 0, \qquad\qquad \text{for }(t,x,y)\in (0,T)\times (0,\infty)\times \mathbb R, \end{split}\\
	U(T,x,y) = U_{T}(x), \qquad\qquad \text{for }(x,y)\in (0,\infty)\times \mathbb R.
	\end{cases}
	\end{equation}
It is easy to see that the expression being maximized in \eqref{hjb} achieves its maximum at the portfolio given by 
	\begin{equation} \label{maxport}
	\pi(t,x,y) = \frac{-\lambda(y)U_{x}(t,x,y) - \rho a(y) U_{xy}(t,x,y)}{\sigma(y)U_{xx}(t,x,y)}.
	\end{equation} 
The optimal strategy is thus $\pi_{t}:= \pi(t,X^{\pi}_{t},Y_{t})$ for $t \in [0,T]$.  Substituting the maximizing portfolio \eqref{maxport} in equation \eqref{hjb} gives 
	\begin{equation} \label{hjb2}
	U_{t} - \frac{1}{2} \frac{(\lambda(y) U_{x} + \rho a(y) U_{xy})^{2}}{U_{xx}} + \frac{1}{2}a^{2}(y)U_{yy} + b(y)U_{y} = 0.
	\end{equation}

\section{Main theorem}\label{shorttimesection}
In this section, we introduce the main result of this paper, that the value function defined in \eqref{valuefunction} can be approximated in such a way as to yield an error measured in terms of time to horizon.  This result is achieved by constructing classical sub- and super-solutions to the HJB equation \eqref{hjb2} which have the form of a second order expansion in powers of $T-t$, the time to horizon.  The sub- and super-solutions will coincide up to the first order terms, and this will serve as the value function approximation.  The second order terms in the expansions of the sub- and super-solution will yield the error. A probabilistic argument using martingale inequalities will show that the value function lies between the constructed  sub- and super-solutions.

	\begin{theorem} \label{maintheorem}
	Let $J(t,x,y)$ be the value function defined in \eqref{valuefunction}, and let $U_{T}(x)$ denote the terminal utility function.  
	Define 
	\begin{equation} \label{valueapproximation}
	\Uh(t,x,y) := U_{T}(x) - (T-t) \frac{\lambda^{2}(y)}{2} \frac{U_{T}'(x)^{2}}{U_{T}''(x)}.
	\end{equation}
Suppose Assumptions \ref{modelassumptions} and \ref{utilityassumptions} hold .Then there exists constants $\C2 > 0$ and $0 < \delta < \min\{1, T\}$ such that 
	\begin{equation}\label{mainresultinequality}
	\left|J(t,x,y) - \Uh(t,x,y) \right| \leq  \C2 (T-t)^{2} h(x), \qquad \text{ for } (t,x,y) \in (T-\delta,T)\times (0,\infty) \times \Bbb R,
	\end{equation}
where $h(x) \equiv 1$ under Case 1 of Assumption \ref{utilityassumptions},
 and $h(x) = x^{1-\alpha} + x^{1-\beta}$ under Case 2 of Assumption \ref{utilityassumptions}; the constants $\C2$ and $\delta$ are independent of $t,x$ and $y$.
	\end{theorem}

	\begin{proof}
	We prove this result in two parts: first, we will construct the classical sub- and super-solutions $\ul{U} (t,x,y)$ and $\ol{U}(t,x,y)$, respectively, to the HJB equation.  Once established, we will then show that the value function lies between the sub- and super-solutions, i.e., $\ul{U}(t,x,y) \leq J(t,x,y) \leq \ol{U}(t,x,y)$.

We begin by constructing $\ul{U}$ and $\ol{U}$.  Consider the HJB equation given in \eqref{hjb2}:
	\begin{equation} \label{shorthjb}
	\begin{cases}
	\begin{split}
	U_{t} + \H(U) = 0, \qquad\qquad \text{for }(t,x,y)\in (0,T)\times (0,\infty)\times \mathbb R, \end{split}\\
	U(T,x,y) = U_{T}(x), \qquad\qquad \text{for }(x,y)\in (0,\infty)\times \mathbb R,
	\end{cases}
	\end{equation} where
	\begin{equation*}
		\begin{split}
		\H(U) := H(y,U,U_{x},U_{y},U_{xx},U_{xy},U_{yy}) &:= - \frac{1}{2} \frac{(\lambda(y) U_{x} + \rho a(y) U_{xy})^{2}}{U_{xx}} + \frac{1}{2}a^{2}(y)U_{yy} + b(y)U_{y}.
		\end{split}
	\end{equation*}
We consider sub- and super-solutions  having the following expansion in terms of powers of $T-t$: 
	\begin{equation*}
	U(t,x,y) := \U0(x,y) + (T-t) \U1(x,y) + (T-t)^{2} \U2(x,y).
	\end{equation*}  
So that our first order approximation coincides with the value function at the terminal time $T$, we choose the terminal condition of \eqref{shorthjb} for the value of $\U0$, i.e.,
	\begin{equation*}
	\U0(x,y) := U_{T}(x) \text{ for all } (x,y) \in (0,\infty) \times \Bbb R.
	\end{equation*}
We now substitute $U$ into \eqref{shorthjb} to obtain:
	\begin{equation}\label{expansion-PDE}\begin{split}
	&-\U1 - 2(T-t) \U2 \\ 
	&-\frac{1}{2} \frac{\left(\lambda(y) (\U0_{x} + (T-t) \U1_{x} + (T-t)^{2} \U2_{x}) + \rho a(y) (\U0_{xy} + (T-t) \U1_{xy} + (T-t)^{2} \U2_{xy}) \right)^{2}}{\U0_{xx} + (T-t) \U1_{xx} + (T-t) \U2_{xx}} \\
	&	 + \frac{1}{2} a^{2}(y) \left(\U0_{yy} + (T-t) \U1_{yy} + (T-t) \U2_{yy} \right) + b(y) \left(\U0_{y} + (T-t) \U1_{y} + (T-t) \U2_{y} \right)  = 0. 
	\end{split}
	\end{equation}
	We choose $\U1$ such that terms of order $O(1)$ on the left-hand side of equation \eqref{expansion-PDE} disappear. To do this, we clear fractions and collect terms of order $O(1)$.  Equating these with zero and then removing any terms equivalent to zero (i.e., terms containing partial derivatives of $\U0$ in the variables $t$ and $y$) yields the following formula for $\U1$: 
	\begin{equation*}
	\U1(x,y) =  -\frac{\lambda^{2}(y)}{2} \frac{U_{T}'(x)^{2}}{U_{T}''(x)}  \text{ for all } (x,y) \in (0,\infty) \times \Bbb R.
	\end{equation*}
Finally, we choose two different functions  for $\U2$ so as to obtain a sub- and super-solution.  This is done by analyzing the formula for $\U2$ which results from equating the coefficients of the $T-t$ terms in \eqref{expansion-PDE}, after clearing fractions, to zero. 
Removing any terms with factors of zero, as before, yields
	\begin{equation}\label{U2formula} 
		\begin{split}
	&\U2(x,y) = \\  &\frac{U_{T}'(x)^{2}}{U_{T}''(x)}\left( \frac{1}{4}\lambda^{4}(y) - \frac{1}{2}b(y) \lambda(y) \lambda'(y) +  \rho a(y) \lambda^{2}(y) \lambda'(y) - \frac{1}{4} a^{2}(y) (\lambda'(y))^{2} - \frac{1}{4} a^{2}(y) \lambda(y) \lambda''(y)  \right) \\ &- \frac{U_{T}'(x)^{3}}{U_{T}''(x)^{3}} \left( \frac{1}{2}\rho a(y) \lambda^{2}(y) \lambda'(y) U_{T}^{(3)}(x) + \frac{\lambda^{4}(y)}{4}\frac{U_{T}'(x)(U_{T}^{(3)}(x))^{2}}{(U_{T}''(x))^{2}} - \frac{\lambda^{4}(y)}{8}\frac{U_{T}'(x)U_{T}^{(4)}(x)}{U_{T}''(x)} \right).
		\end{split}
	\end{equation}
We abbreviate this expression by expanding the terms on the right hand side of \eqref{U2formula} and enumerating the resulting terms as $a_{1}, \dots, a_{8}$, giving \eqref{U2formula} as 
	\begin{equation*}
	\U2(x,y) = \sum \limits_{i = 1}^{8} a_{i}(x,y).
	\end{equation*}
We note that $a_{i} \sim h(x)$, where $h(x) \equiv 1$ if $M(x) = \log(x)$ (i.e., if in Case 1 of Assumption \ref{utilityassumptions}), and $h(x) = x^{1-\alpha} + x^{1-\beta}$ if $M(x) = \frac{x^{1-\alpha}}{1-\alpha}+ \frac{x^{1-\beta}}{1-\beta}$ (i.e., if in Case 2 of Assumption \ref{utilityassumptions}) and bounded in $y$ for $1 \leq i \leq 8$. Set \[\C2 := \left (8\max \limits_{1 \leq i \leq 8} \sup \limits_{ \substack{ x > 0 \\ y \in \Bbb R} } \frac{|a_{i}|}{h(x)} \right ) + 1\] and define $\ol{u}_{2} := \C2 h(x)$ and $\ul{u}_{2} := - \ol{u}_{2}$.  Then substituting 
	\begin{equation}\label{subsolution}
	\ul{U} := \U0 + (T-t) \U1 + (T-t)^{2} \ul{u}_{2}
	\end{equation} 
(and similarly 
	\begin{equation}\label{supersolution}
	\ol{U} := \U0 + (T-t) \U1 + (T-t)^{2} \ol{u}_{2})
	\end{equation} 
in the left-hand side of equation \eqref{shorthjb} and clearing fractions, terms of $O(1)$ disappear, while the terms that comprise the coefficient $T-t$ are strictly positive (respectively, strictly negative).  

We now observe that the following inequality holds: 
	\begin{equation} \label{orderhjbnofrac}
	\ul{U}_{xx}^{2} | \ul{U}_{t} + \H(\ul{U})| \leq c(T-t) \tilde{h}(x),
	\end{equation} 
where $\tilde{h}(x) = x^{-4}$ under Case 1 of Assumption \ref{utilityassumptions}, and $\tilde{h}(x) = (x^{1-\alpha} + x^{1-\beta})(x^{-2 - 2\alpha} + x^{-2 -2\beta})$ under Case 2 of Assumption \ref{utilityassumptions}.  We use Assumption \ref{utilityassumptions} to verify \eqref{orderhjbnofrac}, and we note that \eqref{orderhjbnofrac} holds for $\ol{U}$ in place of $\ul{U}$. Also note that $\frac{1}{c} x^{-2} \leq \ul{U}_{xx} \leq cx^{-2}< 0$ in Case 1 and $\frac{1}{c}(x^{-1 - \alpha} + x^{-1-\beta}) \leq \ul{U}_{xx} \leq c (x^{-1 - \alpha} + x^{-1-\beta}) < 0$ in Case 2 of Assumption \ref{utilityassumptions}, for some $c < 0$. So $(\ul{U}_{xx})^2$ is bounded away from $0$. 

To argue that $\ul{U}$ is a sub-solution to \eqref{shorthjb}, we recall that the coefficient of $T-t$ in the expression $\ul{U}_{xx}^{2}(\ul{U}_{t} + \H(\ul{U}))$ is strictly positive (by choice of $\ul{u}_{2}$).  Inequality \eqref{orderhjbnofrac} implies this coefficient has growth in $x$ on the order of $\tilde{h}(x)$ and growth in $y$ bounded.  \eqref{orderhjbnofrac}  also implies the $o(T-t)$ terms of $\ul{U}_{xx}^{2}(\ul{U}_{t} + \H(\ul{U}))$ also have growth in $x$ on the order of $\tilde{h}(x)$ and growth in $y$ bounded.  Thus, for $t$ near $T$, the positive coefficient of $T-t$ dominates the $o(T-t)$ terms uniformly in $x$ and $y$, implying $\ul{U}_{t} + \H(\ul{U})>0$, i.e., $\ul{U}$ is a classical sub-solution of \eqref{shorthjb}.  A mirror of this argument proves $\ol{U}$ is a classical super-solution of \eqref{shorthjb}.

It remains to be shown that the value function $J(t,x,y)$ given in $\eqref{valuefunction}$ lies between the sub- and super-solution, i.e., 
	\begin{equation*}
	\ul{U}(t,x,y) \leq J(t,x,y) \leq \ol{U}(t,x,y) \qquad \text{ for all } (t,x,y) \in [0,T] \times (0,\infty) \times \Bbb R.
	\end{equation*}  
We will first show that $\ul{U}(t,x,y) \leq J(t,x,y)$, and then we will show $J(t,x,y) \leq \ol{U}(t,x,y)$.
	To prove $\ul{U}(t,x,y) \leq J(t,x,y)$, we first consider the trading strategy $\ul{\pi}(t,X^{\ul{\pi}}_{t},Y_{t})$ generated by the sub-solution $\ul{U}$, where $\ul{\pi}(t,x,y)$ is the function obtained by substituting $\ul{U}$ into \eqref{maxport}, i.e., 
	\begin{equation*}
		\ul{\pi}(t,x,y) = \frac{-\lambda(y)\ul{U}_{x}(t,x,y) - \rho a(y) \ul{U}_{xy}(t,x,y)}{\sigma(y)\ul{U}_{xx}(t,x,y)}.
	\end{equation*}
Applying Ito's formula to $\ul{U}(t,X^{\ul{\pi}}_{t},Y_{t})$ gives
	\begin{equation} \label{Itosubsolution}
		\begin{split}
		\underbrace{\ul{U}(T, X^{\ul{\pi}}_{T}, Y_{T})}_{U_{T}(X^{\ul{\pi}}_{T})} - \ul{U}(t, X^{\ul{\pi}}_{t}, Y_{t}) =& \int \limits_{t}^{T} \left ( \ul{U}_{t} +  \sigma \ul{\pi} \lambda \ul{U}_{x} +  b\ul{U}_{y} + \frac{1}{2}  \sigma^{2} \ul{\pi}^{2} \ul{U}_{xx} +  \sigma \ul{\pi} a \rho \ul{U}_{xy} + \frac{1}{2}  a^{2} \ul{U}_{yy} \right ) \,ds \\
		+& \underbrace{\int \limits_{t}^{T} \left (  \sigma \ul{\pi}\ul{U}_{x} +  a \rho \ul{U}_{y} \right ) \, dW^{1}_{s} + \int \limits_{t}^{T}  a \sqrt{1 - \rho^{2}}\ul{U}_{y} \, dW^{2}_{s}}_{\text{local martingales}}.
		\end{split}
	\end{equation}
Since the stochastic integrals on the right hand side are local martingales, we can find a sequence $\{ \tau_{n} \}_{n = 1}^{\infty}$ of stopping times such that $\tau_{n} \in [t,T]$, $\tau_{n} \leq \tau_{n + 1}$ a.s. for all $n$, and $\tau_{n} \to T$ a.s. as $n \to \infty$.  In particular, if we replace $T$ with $T \wedge \tau_{n}$, the local martingales will become martingales:
	\begin{equation}\label{localizedIto}
		\begin{split}
		&\ul{U}(T\wedge \tau_{n}, X^{\ul{\pi}}_{T \wedge \tau_{n}}, Y_{T\wedge \tau_{n}}) - \ul{U}(t, X^{\ul{\pi}}_{t}, Y_{t}) \\
		&=\int \limits_{t }^{T\wedge \tau_{n}} \left ( \ul{U}_{t} +  \sigma \ul{\pi} \lambda \ul{U}_{x} +  b\ul{U}_{y} + \frac{1}{2}  \sigma^{2} \ul{\pi}^{2} \ul{U}_{xx} +  \sigma \ul{\pi} a \rho \ul{U}_{xy} + \frac{1}{2}  a^{2} \ul{U}_{yy} \right )  \,ds \\& +  \int \limits_{t}^{T\wedge \tau_{n}} \left (  \sigma \ul{\pi}\ul{U}_{x} +  a \rho \ul{U}_{y} \right ) \, dW^{1}_{s} + \int \limits_{t}^{T\wedge \tau_{n}}  a \sqrt{1 - \rho^{2}}\ul{U}_{y} \, dW^{2}_{s}.
		\end{split}
	\end{equation}
Note that the integrand of the first term on the right hand side of \eqref{localizedIto} is the left-hand side of the  HJB equation with the sub-solution $\ul{U}$ substituted in, thus making the term non-negative.  Taking the conditional expectation of both sides of equation \eqref{localizedIto} we get  \[\ul{U}(t, x, y) \leq E[ \ul{U}(T\wedge \tau_{n}, X^{\ul{\pi}}_{T \wedge \tau_{n}}, Y_{T\wedge \tau_{n}}) | X^{\ul{\pi}}_{t} = x, Y_{t} = y].\]

Now, clearly, $\ul{U}(T \wedge \tau_{n}, X^{\ul{\pi}}_{T\wedge \tau_{n}}, Y_{T\wedge \tau_{n}}) \to \ul{U}(T, X^{\ul{\pi}}_{T}, Y_{T}) = U_{T}(X^{\ul{\pi}}_{T})$ a.s. as $n \to \infty$.  Also, we have
	\begin{equation}\label{dominating_subsol}
		\begin{split}
		&|\ul{U}(T \wedge \tau_{n}, X^{\ul{\pi}}_{T\wedge \tau_{n}}, Y_{T\wedge \tau_{n}})| \\
		&= \left |U_{T}(X^{\ul{\pi}}_{T\wedge \tau_{n}}) - (T-T \wedge \tau_{n}) \frac{\lambda^{2}(Y_{T\wedge \tau_{n}})}{2} \frac{U_{T}'(X^{\ul{\pi}}_{T\wedge \tau_{n}})^{2}}{U_{T}''(X^{\ul{\pi}}_{T\wedge \tau_{n}})}  - c_{2}(T\wedge\tau_{n} - t)^{2} h(X^{\ul{\pi}}_{T\wedge\tau_{n}})\right | \\
		&\leq \left |U_{T}(X^{\ul{\pi}}_{T\wedge \tau_{n}}) \right | + T \frac{\lambda^{2}(Y_{T\wedge \tau_{n}}))}{2} \left |\frac{U_{T}'(X^{\ul{\pi}}_{T\wedge \tau_{n}})^{2}}{U_{T}''(X^{\ul{\pi}}_{T\wedge \tau_{n}})} \right | + c_{2}Th(X^{\ul{\pi}}_{T \wedge \tau_{n}}) \\
		& \leq \C3 G(X^{\ul{\pi}}_{T\wedge \tau_{n}}),
		\end{split}
	\end{equation}
for some constant $\C3$, where $G(x) = \log(x) + 1$ under Case 1 of Assumption \ref{utilityassumptions}, and $G(x) = x^{1-\alpha} + x^{1-\beta}$ under Case 2 of Assumption \ref{utilityassumptions}.

To see that $\{G(X^{\ul{\pi}}_{T\wedge \tau_{n}})\}_{n =1}^{\infty}$ is dominated by an integrable random variable, we refer the reader to Lemma \ref{boundedintfunc} which is proved in the appendix.  Thus, we are in a position to apply the dominated convergence theorem, which yields  
	\begin{equation*}
	E[\ul{U}(T\wedge \tau_{n}, X^{\ul{\pi}}_{T \wedge \tau_{n}}, Y_{T\wedge \tau_{n}})| X^{\ul{\pi}}_{t} = x, Y_{t} = y ] \to E[ U_{T}(X^{\ul{\pi}}_{T}) | X^{\ul{\pi}}_{t} = x, Y_{t} = y ] \ a.s.
	\end{equation*}
	as $n \to \infty$.
This implies \[\ul{U}(t, x, y) \leq E[ U_{T}(X^{\ul{\pi}}_{T}) | X^{\ul{\pi}}_{t} = x, Y_{t} = y ].\]  From the admissibility of $\ul{\pi}(t,X^{\ul{\pi}}_{t},Y_{t})$ (which follows from a similar argument used to prove Lemma \ref{admissibilitypihat}), it immediately follows that $\ul{U}(t,x,y) \leq J(t,x,y)$.

We now verify that $J(t,x,y) \leq \ol{U}(t,x,y)$.  To begin, let $\tilde{\pi}$ be any admissible trading strategy.  Note that because $\ol{U}$ is a super-solution of the HJB equation \eqref{hjb}, we have that 
	\begin{equation*}
\ol{U}_{t} +   \frac{1}{2} \sigma^{2}(y)\tilde{\pi}^{2}\ol{U}_{xx} + \tilde{\pi}(\sigma(y) \lambda(y) \ol{U}_{x} + \rho \sigma(y) a(y) \ol{U}_{xy})  + \frac{1}{2}a^{2}(y)\ol{U}_{yy} + b(y)\ol{U}_{y} \leq 0.
	\end{equation*}
Thus, applying Ito's formula to $\ol{U}(t,X^{\tilde{\pi}}_{t},Y_{t})$, followed by localizing and taking conditional expectations, we have \[E[ \ol{U}(T\wedge \tau_{n}, X^{\tilde{\pi}}_{T \wedge \tau_{n}}, Y_{T\wedge \tau_{n}}) | X^{\tilde{\pi}}_{t} = x, Y_{t} = y] \leq \ol{U}(t, x, y) \text{ for each }n.\]  As in \eqref{dominating_subsol}, we can show that $|\ol{U}(T\wedge \tau_{n}, X^{\tilde{\pi}}_{T \wedge \tau_{n}}, Y_{T\wedge \tau_{n}})|\leq \C3 G(X^{\tilde{\pi}}_{T\wedge \tau_{n}})$ where $G(X^{\tilde{\pi}}_{T\wedge \tau_{n}})$ is dominated by an integrable random variable (shown in Lemma \ref{boundedintfunc}). The dominated convergence theorem then implies $E[ U_{T}(X^{\tilde{\pi}}_{T}) | X^{\tilde{\pi}}_{t} = x, Y_{t} = y ] \leq \ol{U}(t,x,y)$.  As $\tilde{\pi}_{t}$ is an arbitrary admissible portfolio, this implies that $J(t,x,y) \leq \ol{U}(t,x,y)$, as desired. 

	Having established that the value function lies between the sub- and super-solution, i.e., $\ul{U}(t,x,y) \leq J(t,x,y) \leq \ol{U}(t,x,y)$, it follows immediately from the definitions of $\ul{U}$ and $\ol{U}$ in \eqref{subsolution} and \eqref{supersolution}, respectively,  that for $\Uh(t,x,y)$ as in \eqref{valueapproximation}, $|J(t,x,y) - \Uh(t,x,y) | \leq c_{2}(T-t)^{2}h(x)$.
	\end{proof}

\section{Building the approximating portfolio} \label{approxportsec}
In section \ref{shorttimesection}, we showed that the value function $J(t,x,y)$ given in \eqref{valuefunction} can be approximated by a first order expansion in powers of the time to horizon $T-t$, namely, $\hat{U}(t,x,y)$ given in \eqref{valueapproximation}.  In addition, we showed that the error between $J(t,x,y)$ and $\hat{U}(t,x,y)$ is controlled by the square time to horizon $(T-t)^{2}$.  

In this section, we will show that our first order approximation generates a close-to-optimal trading strategy near horizon.  To show this, we first recall formula \eqref{maxport}, which the Verification Theorem tells us would represent the optimal trading strategy in the case that the HJB equation \eqref{hjb} were well-posed.  As we do not assume a classical solution to the HJB equation, the conclusion of the Verification Theorem may not be applied.  Formula \eqref{maxport} will still be useful in our analysis, however. We will show that our smooth approximation $\hat{U}(t,x,y)$, when substituted into \eqref{maxport}, produces a portfolio which yields an expected utility close to the maximum expected utility, with the error measured in terms of the square time to horizon $(T-t)^{2}$.  This result is stated in the following lemma.

	\begin{lemma} \label{utilityapproxlemma}
Let $J(t,x,y)$ be the value function defined in \eqref{valuefunction}, $\hat{U}(t,x,y)$ be as in \eqref{valueapproximation}, and let $X^{\pih}_{s}$ be the wealth process with evolution described by \eqref{dscw} under portfolio $\pih(s,X^{\pih}_{s},Y_{s})$, where $\pih(s,x,y)$ the function given by \begin{equation}\label{approxoptimalport}
		\pih(s,x,y) := -\frac{\lambda(y)}{\sigma(y)} \frac{\Uh_{x}(s,x,y)}{\Uh_{xx}(s,x,y)} - \frac{\rho a(y)}{\sigma(y)} \frac{\Uh_{xy}(s,x,y)}{\Uh_{xx}(s,x,y)}, \,\,\,\,\,\,\, s\in[t,T], \,\,x \in (0,\infty), \,\, y \in \Bbb R.
	\end{equation}
Under Assumptions \ref{modelassumptions} and \ref{utilityassumptions}, there exists a constant $C > 0$ and $0 < \delta < \min\{1,T\}$ such that
		\begin{equation*}
		|J(t,x,y) - E[ U_{T}(X^{\pih}_{T}) |X^{\pih}_{t} = x, Y_{t} = y ]| \leq C (T-t)^{2}h(x), \qquad \text{ for } (t,x,y) \in (T-\delta,T)\times (0,\infty) \times \Bbb R
		\end{equation*}
where $h(x) \equiv 1$ under Case 1 of Assumption \ref{utilityassumptions}, and $h(x) = x^{1-\alpha} + x^{1-\beta}$ under Case 2 of Assumption \ref{utilityassumptions}; the constants $C$ and $\delta$ are independent of $t$, $x$ and $y$.
	\end{lemma}

\begin{proof}
We begin by referring the reader to Lemma \ref{admissibilitypihat}, which is proven in the Appendix and asserts the strategy $\pih_{t}$ is indeed admissible as in Definition \ref{admissibledefinition}.  

We now prove that the expected utility of terminal wealth under $\pih_{t}$ is near the maximal expected utility.  To do this, we start by applying Ito's formula to $\Uh(s, X^{\pih}_{s}, Y_{s})$ and obtain
	\begin{equation} \label{ItoUhat}
		\begin{split}
		\Uh(T, X^{\pih}_{T}, Y_{T}) - \Uh(t, X^{\pih}_{t}, Y_{t}) =& \int \limits_{t}^{T} \left ( \Uh_{t} +  \sigma \pih \lambda \Uh_{x} +  b\Uh_{y} + \frac{1}{2}  \sigma^{2} \pih^{2} \Uh_{xx} +  \sigma \pih a \rho \Uh_{xy} + \frac{1}{2}  a^{2} \Uh_{yy} \right ) \,ds \\
		+& \int \limits_{t}^{T} \left (  \sigma \pih\Uh_{x} +  a \rho \Uh_{y} \right ) \, dW^{1}_{s} + \int \limits_{t}^{T}  a \sqrt{1 - \rho^{2}}\Uh_{y} \, dW^{2}_{s}.
		\end{split}
	\end{equation}
Recall that by definition of $\Uh$ and Assumption \ref{utilityassumptions}, we get $|\partial_{t} \Uh + \H(\Uh)| = O(T-s)O(h(X^{\pih}_{s}))$, which shows the integrand of the drift term is $O(T-s)O(h(X^{\pih}_{s}))$.  Parallel to the proof of Theorem \ref{maintheorem}, we use the sequence of stopping times $\{\tau_{n}\}_{n=1}^{\infty}$ to localize \eqref{ItoUhat}, converting the local martingale terms to martingales.  Taking the conditional expectation of both sides then yields
	\begin{equation*}
		\begin{split} 
		&E[\Uh(T \wedge \tau_{n}, X^{\pih}_{T \wedge \tau_{n}}, Y_{T\wedge \tau_{n}})| X^{\pih}_{t} = x, Y_{t} = y ] - \Uh(t, x,y) \\
		&=\int \limits_{t }^{T\wedge \tau_{n}} E \left [ O(T-s)O(h(X^{\pih}_{s}))  | X^{\pih}_{t} = x, Y_{t} = y  \right ] \,ds.
		\end{split}
	\end{equation*}

Using the uniform bound of $h(X^{\pih})$ which can be obtained from the proof of Lemma \ref{admissibilitypihat}, it follows that 
	\begin{equation} \label{stoppedbound}
	|E[\Uh(T \wedge \tau_{n}, X^{\pih}_{T \wedge \tau_{n}}, Y_{T\wedge \tau_{n}})| X^{\pi}_{t} = x, Y_{t} = y ] - \Uh(t, x, y)| \leq C_{1}(T-t)^{2}h(x),
	\end{equation}
for some constant $C_1>0$. 
Now, parallel to the proof of Theorem \ref{maintheorem}, we have $\Uh(T \wedge \tau_{n}, X^{\pih}_{T\wedge \tau_{n}}, Y_{T\wedge \tau_{n}}) \to \Uh(T, X^{\pih}_{T}, Y_{T}) = U_{T}(X^{\pih}_{T})$ a.s. as $n \to \infty$.  Also, we have $|\Uh(T \wedge \tau_{n}, X^{\pih}_{T\wedge \tau_{n}}, Y_{T\wedge \tau_{n}})|  \leq \C3 G(X^{\pih}_{T\wedge \tau_{n}})$ for some constant $\C3$, where $G(x) = \log(x) + 1$ under Case 1 of Assumption \ref{utilityassumptions}, and $G(x) = x^{1-\alpha} + x^{1-\beta}$ under Case 2 of Assumption \ref{utilityassumptions}.

By Lemma \ref{boundedintfunc}, the sequence $\{G(X^{\pih}_{T\wedge \tau_{n}})\}_{n =1}^{\infty}$ is dominated by an integrable function, implying by the dominated convergence theorem that 
	\begin{equation*}
	E[\Uh(T\wedge \tau_{n}, X^{\pih}_{T \wedge \tau_{n}}, Y_{T\wedge \tau_{n}})| X^{\pih}_{t} = x, Y_{t} = y ] \to E[ U_{T}(X^{\pih}_{T}) | X^{\pih}_{t} = x, Y_{t} = y ] \text{ a.s. as } n \to \infty.
	\end{equation*}
As a result, \eqref{stoppedbound} gives 
	\begin{equation} \label{notstoppedbound}
	|E[U_{T}(X^{\pih}_{T})| X^{\pi}_{t} = x, Y_{t} = y ] - \Uh(t, x, y)| \leq C_{1}(T-t)^{2}h(x).
	\end{equation}
By Theorem \ref{maintheorem} and inequality \eqref{notstoppedbound}, it follows that  
	\begin{equation*}
	\begin{split}
	|E[U_{T}(X^{\pih}_{T})| X^{\pi}_{t} = x, Y_{t} = y ] - J(t, x, y)| &\leq |E[U_{T}(X^{\pih}_{T})| X^{\pi}_{t} = x, Y_{t} = y ] - \Uh(t, x, y)| + |\Uh(t, x,y) - J(t, x, y)| \\
	&\leq C_{1}(T-t)^{2} h(x) + c_{2} (T-t)^{2} h(x) \\
	&\leq C(T-t)^{2}h(x),
	\end{split}
	\end{equation*}
for $0 < T-t < \delta$.
\end{proof}

\section{Example} \label{smalltimeexamplesec}
We consider the following stochastic volatility model, which was used in \cite{Fouque2015Portfolio}, with parameter estimations taken from \cite{ChackoViceira05Dynamic}.  The time horizon considered is $[0,T]$.  The risky asset satisfies \eqref{stockprice} with $\mu(y) = \mu$ a constant function and $\sigma(y) = \frac{1}{\sqrt{y}}$.  The stochastic factor satisfies \eqref{stochasticfactor} with $b(y) = (m - y)$ and $a(y) = \beta \sqrt{y}$ where $m$ and $\beta$ are constants.  In \cite{Fouque2015Portfolio}, the authors assume their model has a slow factor (hence the presence of a factor $\delta$ in their model).  As we do not assume the factor in our model is a slow factor, we have set $\delta = 1$.  We set: $\mu = 0.0811$, $m = 27.9345$, and $\beta = 1.12$; the variable $y$ is fixed at $27.9345$; $T=2$; the correlation coefficient between our two Brownian motions is $\rho = 0.5241$; and $\gamma = 3$.  Under this model, we consider the power utility function $U_{T}(x) = \frac{x^{(1-\gamma)}}{1-\gamma} = -\frac{1}{2x^{2}}$.  Note that while this utility function satisfies Assumption \ref{utilityassumptions}, not all of the model assumptions in Assumption \ref{modelassumptions} are satisfied (e.g., $\lambda(y)$ is not absolutely bounded).  Nevertheless, our results are shown in this section to be good approximations under this model as well.

The authors in \cite{Fouque2015Portfolio} obtained an explicit formula for the value function under the assumed model by solving a linear PDE derived in \cite{Zariphopoulou01Val}. We now restate the formula for the value function found in \cite{Fouque2015Portfolio}.  If $f(r) := \frac{ \beta^{2}}{2} r^{2} + (\frac{(1-\gamma) \beta \mu \rho - \gamma}{\gamma})r + \frac{(\gamma + (1-\gamma)\rho^{2})(1-\gamma)\mu^{2}}{2\gamma^{2}}$, where we substitute the above values we have assumed for the variables in $f(r)$, then solving $f(r) =0$ gives one positive root and one negative root of $f$, denoted $a_{+}$ and $a_{-}$, respectively.  In addition, we set $\alpha$ to be the square root of the discriminant of the quadratic polynomial $f(r)$.  Then, if we define $A(t,T) := \frac{(1 - e^{-\alpha (T-t)})a_{-}}{1- \frac{a_{-}}{a_{+}}e^{-\alpha(T-t)}}$ and $B(t,T) :=  m \left(  (T-t)a_{-} - \frac{2 }{ \beta^{2}} \log \left(\frac{1 - \frac{a_{-}}{a_{+}}e^{-\alpha(T-t)}}{1 - \frac{a_{-}}{a_{+}}} \right) \right )$, the value function is given by
	\begin{equation} \label{exvalueformula}
		U(t,x,y) = -\frac{1}{2x^{2}} e^{\left (\frac{\gamma}{\gamma + (1-\gamma) \rho^{2}} \right ) ( yA(t,T) + B(t,T)  ) }.
	\end{equation}

Recall that our approximation of the value function is given by
	\begin{equation} \label{exapproxformula}
		\hat{U}(t,x,y) = U_{T}(x) - (T-t) \frac{\lambda^{2}(y)}{2} \frac{U_{T}'(x)^{2}}{U_{T}''(x)}, 
	\end{equation}
We can now substitute \eqref{exvalueformula} and \eqref{exapproxformula} into \eqref{maxport} to obtain the optimal and approximating portfolios, $\pi_{U}$ and $\pih$, respectively.  For the parameter values assumed in the beginning of this section, we obtain the following formulas for the value function and its approximation, the optimal portfolio and its approximation, and the respective errors: 
\begin{table}[H]
\caption{\label{examplesectable}}
	\begin{tabularx}{47em}{ c  c  c  c  c  c  c  c } 
	\hline
	\parbox[c]{1em}{\vspace*{6mm}} $t$ & $T$ & $U(t,x,y)$ & $\Uh(t,x,y)$ & $|U - \Uh|$ & $\pi_{U}(t,x,y)$ & $\pih(t,x,y)$ & $|\pi_{U} - \pih|$ \\
	\hline
	\parbox[c]{1em}{\vspace*{10mm}} $1.5$ & $2$ & $\approx -\dfrac{0.485022}{x^{2}}$ & $\approx -\dfrac{0.484689}{x^{2}}$ & $\approx \dfrac{0.000333}{x^{2}}$ & $\approx 0.750482x$ & $\approx 0.748982x$ & $\approx 0.0015x $\\
\parbox[c]{1em}{\vspace*{10mm}} 1.9 & 2 & $\approx -\dfrac{0.496952}{x^{2}}$ & $\approx -\dfrac{0.496938}{x^{2}}$ & $\approx \dfrac{0.000014}{x^{2}}$ & $\approx 0.754024x$ & $\approx 0.753957x$ & $\approx 0.000067x$ \\
\hline
\end{tabularx}
\end{table}
In figures \ref{fig:valuefullonehalf}, \ref{fig:valuetimeonehalf}, and \ref{fig:valuetimepointone}, we graph the value function against the zero and first order approximations.  In figures \ref{fig:portfoliosonehalf} and \ref{fig:portfoliospointone}, we graph the optimal portfolio against our approximating portfolio.
	\begin{figure}[H]
	\begin{center}
		\begin{minipage}[t]{.915\linewidth}
		\includegraphics[scale=.7]{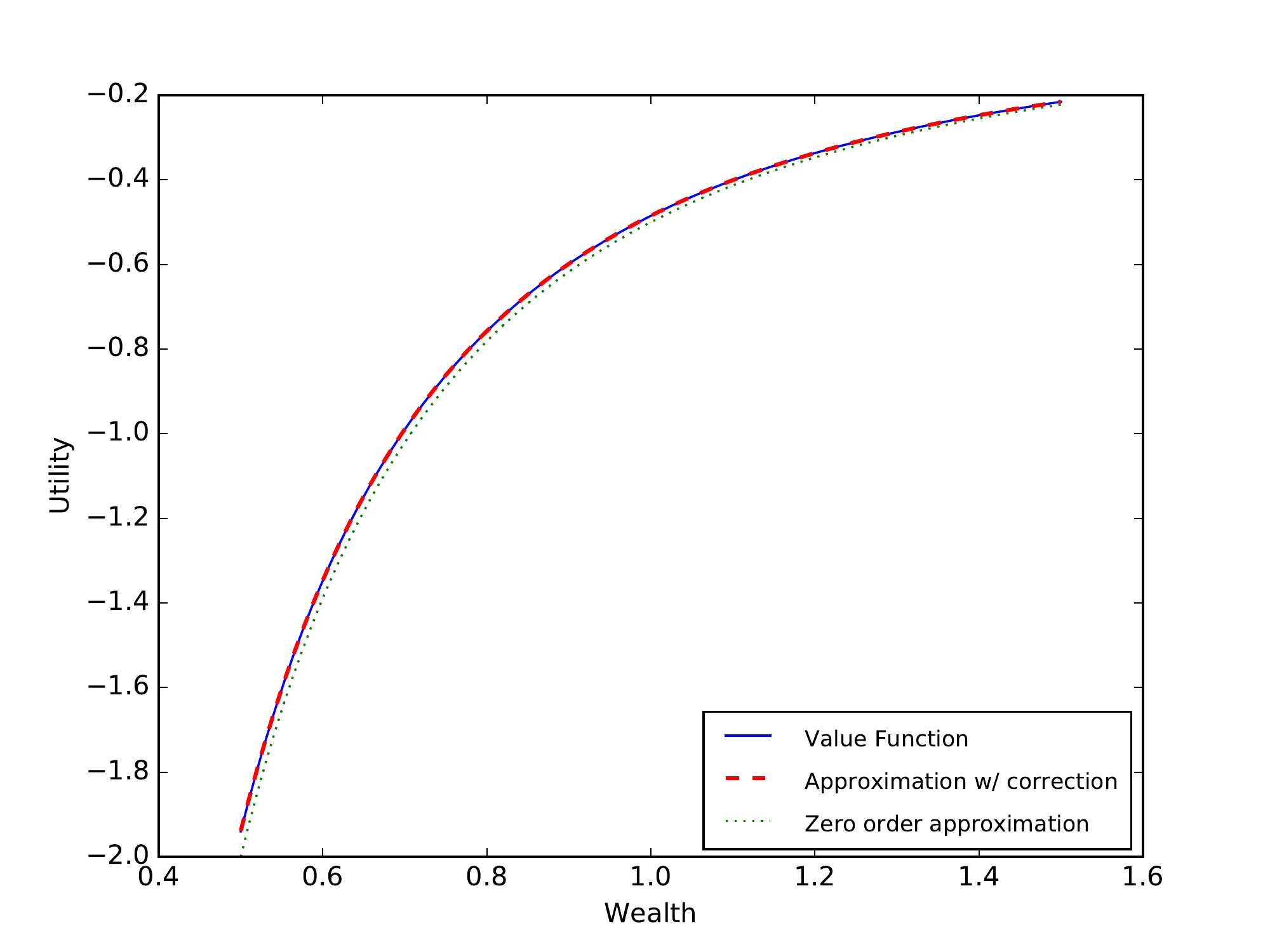} 
		\caption{($t = 1.5, \,T=2$) The value function is plotted against the zero order approximation, $U_{T}(x)$, and the zero order approximation with the additional correction term.  It is difficult to distinguish between the value function and the first order approximation (i.e.\ approximation with correction term).}
		\label{fig:valuefullonehalf}
		\end{minipage}
	\end{center}
	\end{figure}
	
	\begin{figure}[H]
	\centering
		\begin{minipage}[t]{.48\linewidth}
		\includegraphics[scale=0.43]{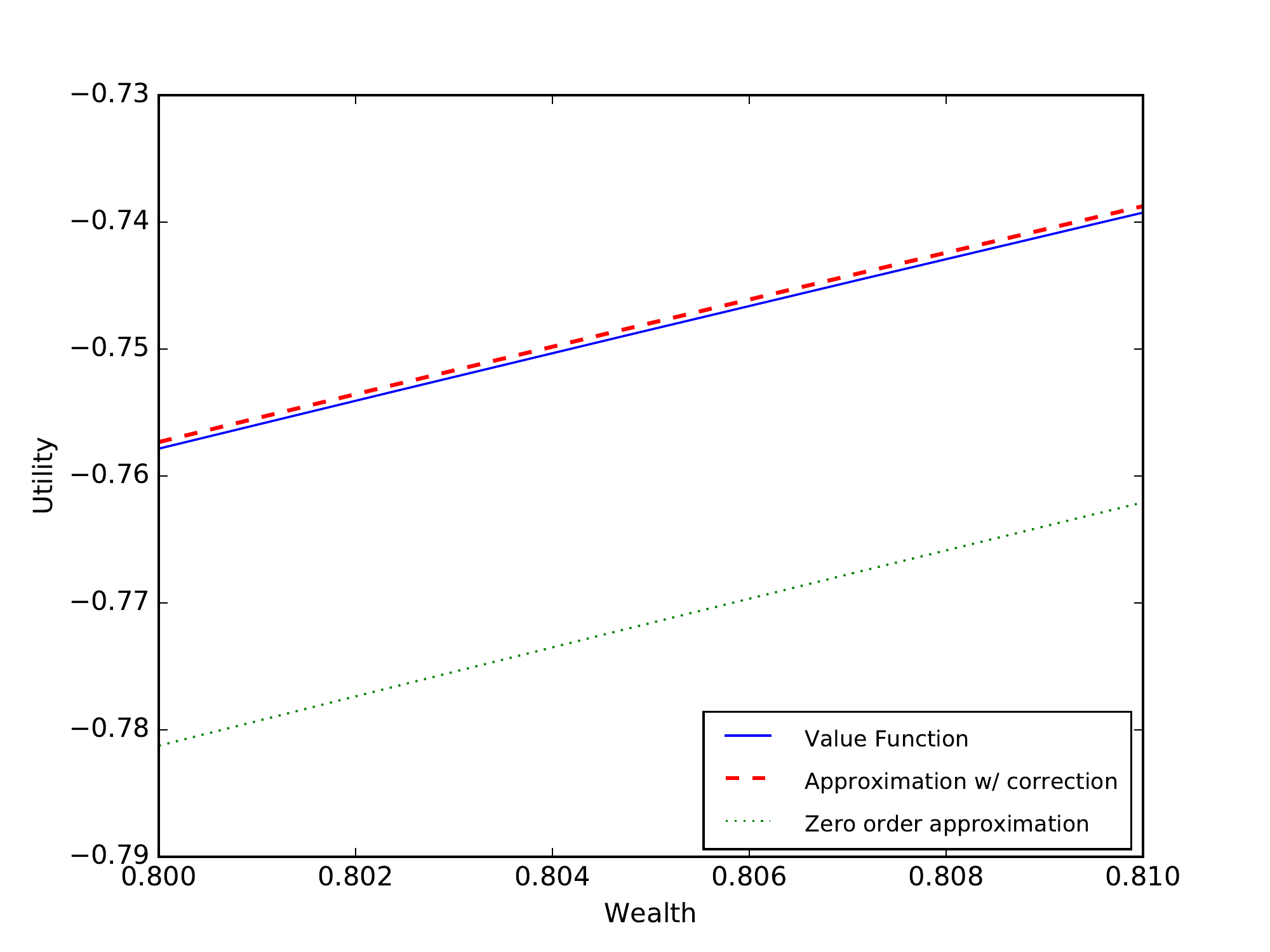}
		\caption{($t = 1.5, \,T=2$) When figure \ref{fig:valuefullonehalf} is zoomed in over a shorter wealth interval, differences between the value function and the approximations are more apparent. }
		\label{fig:valuetimeonehalf}
		\end{minipage}
		\hfill
		\begin{minipage}[t]{.48\linewidth}
		\includegraphics[scale=0.43]{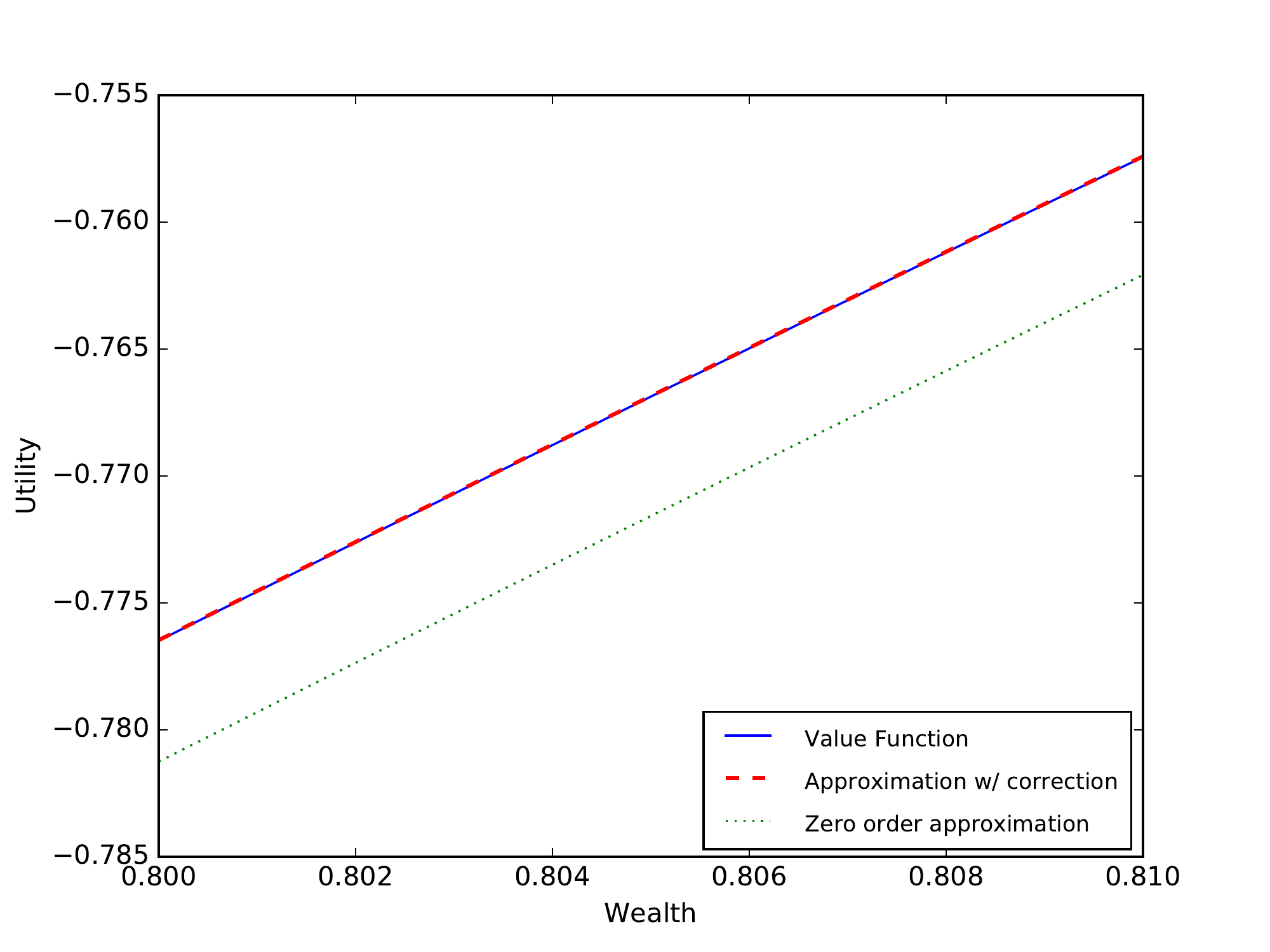}
		\caption{($t=1.9, \,T=2$) When the time interval is shortened from a length of $0.5$ to a length of $0.1$, the approximation with correction is much closer to the value function.}
		\label{fig:valuetimepointone}
		\end{minipage}
	\end{figure}
	\begin{figure}[H]
		\begin{minipage}[t]{.48\linewidth}
		\includegraphics[scale=0.43]{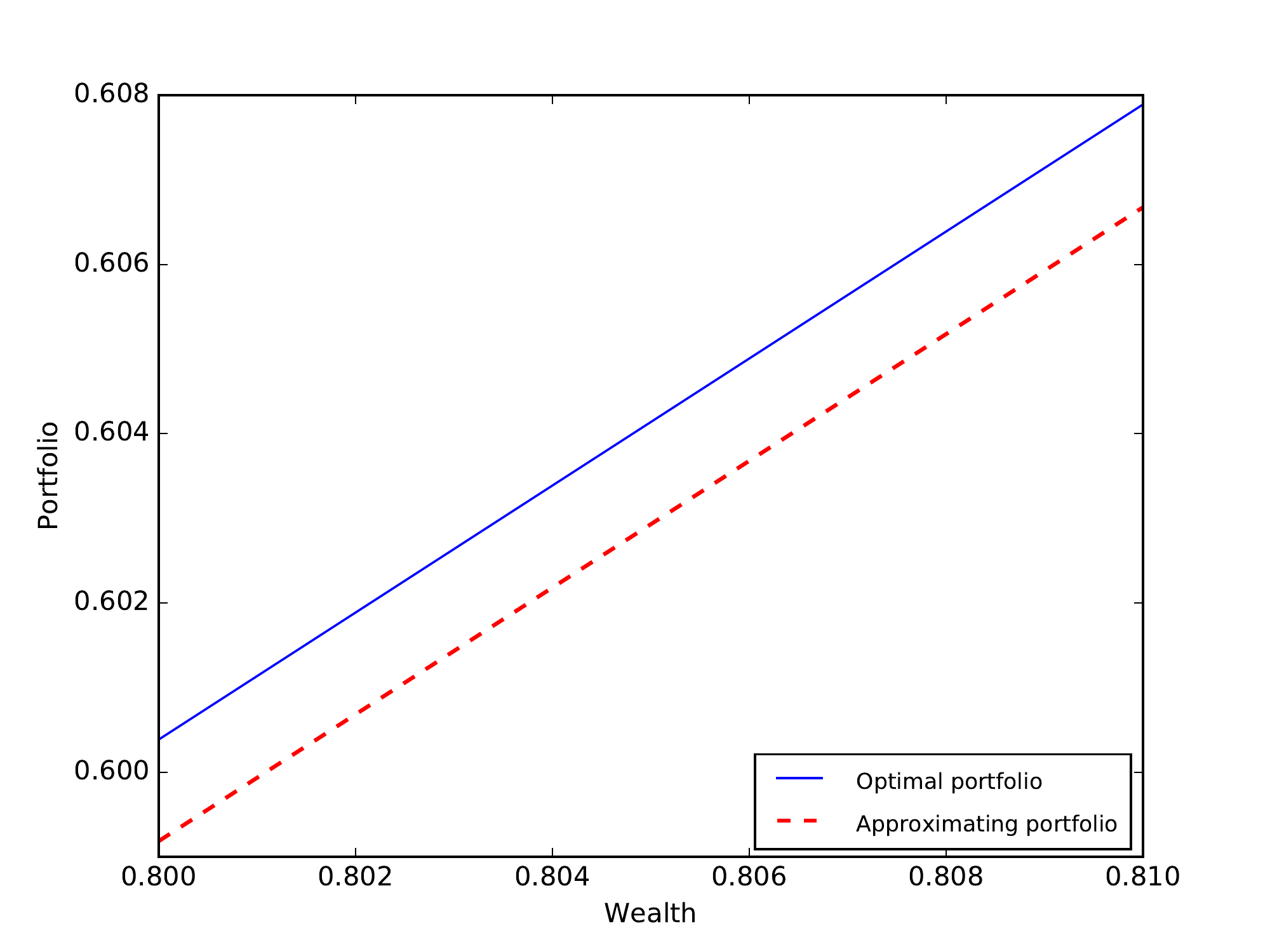}
		\caption{($t=1.5, \,T=2$) The portfolios generated by the value function and the approximation with correction are shown in this figure to be close.}
		\label{fig:portfoliosonehalf}
		\end{minipage}
		~~~
		\begin{minipage}[t]{.48\linewidth}
		\includegraphics[scale=0.43]{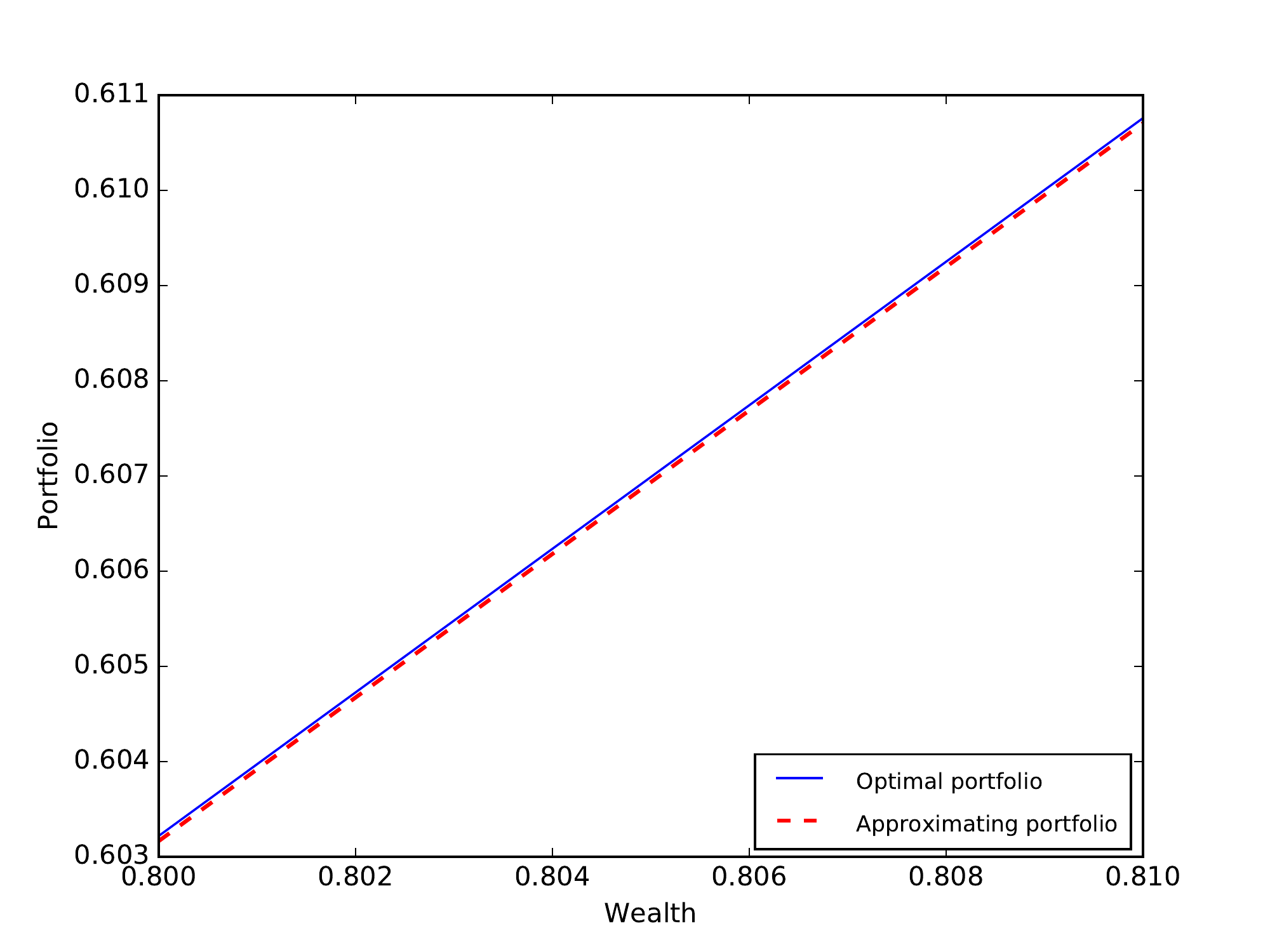}
		\caption{($t=1.9, \,T=2$) When the time interval is shortened from a length of $0.5$ to a length of $0.1$, the approximating portfolio is much closer to the optimal portfolio.}
		\label{fig:portfoliospointone}
		\end{minipage}
	\end{figure}

\section{Portfolio optimization on a finite time horizon}
\subsection{Approximation scheme} \label{schemesec}
In section \ref{shorttimesection}, we approximated the value function $J(t,x,y)$, given in \eqref{valuefunction}, for values of time $t$ near the terminal time $T$.  We then used this approximation in section \ref{approxportsec} to generate a trading strategy $\pih_{t} := \pih(t,X^{\pih}_{t}, Y_{s})$ (with the function $\pih(t,x,y)$ given in \eqref{approxoptimalport}) which was shown to be close-to-optimal when the time to horizon $T-t$ is small.  

In this section, we present a heuristic scheme to approximate the value function for all times $t$ in some finite horizon $[0,T]$, and then utilize this approximation in tandem with the function $\pi(t,x,y)$ from \eqref{maxport} to generate a close-to-optimal trading strategy on $[0,T]$.  To begin, we partition the interval $[0,T]$ into small sub-intervals, given by $\{0 = t_{0} < t_{1} < \dots < t_{n - 1} < t_{n} = T\}$. The scheme is then given by
	\begin{equation} \label{valueapproximationscheme}
	\begin{split}
\Uh(t,x,y) := \Uh(t_{k+1},x,y) + (t_{k+1}-t) \biggl [ - \frac{1}{2}\frac{(\lambda(y) \Uh_{x}(t_{k+1},x,y) + \rho a(y)  \Uh_{xy}(t_{k+1},x,y))^{2}}{\Uh_{xx}(t_{k+1},x,y)}    \\
+ \frac{1}{2} a^{2}(y) \Uh_{yy}(t_{k+1},x,y) + b(y)  \Uh_{y}(t_{k+1},x,y) \biggr],
	\end{split}
	\end{equation}
for $t_{k} \leq t \leq t_{k+1}$, $(x,y) \in (0,\infty) \times \Bbb R$, with $\Uh(T,x,y) = U_{T}(x)$.  The close-to-optimal trading strategy will then be given by $\pih_{t} := \pih(t,X^{\pih}_{t},Y_{t})$, where the function $\pih(t,x,y)$ is the function
	\begin{equation} \label{finitehorizonpi}
	\pih(t,x,y) := \frac{-\lambda(y)\Uh_{x}(t,x,y)}{\sigma(y)\Uh_{xx}(t,x,y)} - \frac{\rho a(y) \Uh_{xy}(t,x,y)}{\sigma(y)\Uh_{xx}(t,x,y)},
	\end{equation}
for $t_{k} \leq t < t_{k+1}$, $(x,y) \in (0,\infty) \times \Bbb R$, with $\Uh(T,x,y) = U_{T}(x)$.

A formal justification of these formulae is given as follows.  Approximating formula \eqref{valueapproximationscheme} is obtained by recursively implementing the technique for constructing the sub- and super-solutions to the HJB equation in section \ref{shorttimesection}.  We first apply the result in Theorem \ref{maintheorem} to the horizon $[t_{n-1},T]$.  We remind the reader that the zero order term in the approximation formula will be the terminal condition $U_{T}(x)$ on the interval $[t_{n-1},T]$.  For the earlier time interval, say, $[t_{n-2}, t_{n-1}]$, $\ul{U}(t_{n-1},x,y)$ and $\ol{U}(t_{n-1},x,y)$ will serve as the terminal conditions for the sub- and super-solutions constructed on the interval $[t_{n-2}, t_{n-1}]$.  Note the $y$-independence of the terminal condition $U_{T}(x)$ on the subinterval $[t_{n-1},T]$ reduces formula \eqref{valueapproximationscheme} to \eqref{valueapproximation}.  The dependence of the terminal conditions on $y$ at earlier time intervals introduces the additional terms in the first order term of \eqref{valueapproximationscheme}.

\begin{remark} The accuracy of approximations \eqref{valueapproximationscheme} and \eqref{finitehorizonpi} will be rigorously proved in future work.  This can be accomplished by repeating the procedure used to verify the accuracy in section \ref{shorttimesection}.  On a fixed, finite horizon, however, this will require a higher degree of regularity of the terminal condition.  In addition, establishing an inequality in the spirit of \eqref{orderhjbnofrac} will be more involved and is beyond the scope of this paper. \end{remark}

\subsection{Example}
In this section, we consider the model and utility function described in Section \ref{smalltimeexamplesec}, and graphically analyze the accuracy of approximation of our scheme \eqref{valueapproximationscheme} on the finite horizon $[0,T]$, where $T=2$.  In Section \ref{smalltimeexamplesec}, the value function was calculated at times $t =1.5$ and $t = 1.9$, which were close to T = 2.  However, following the approximation scheme in Section \ref{schemesec}, we can now approximate the value function at time $t = 0$.  

For comparison, we also compute a Merton approximation to the optimal portfolio. A naive Merton-like approximation for the value function is given by 
	\begin{equation*}\label{mertonscheme}
\Umer(x):= -e^{-0.0001569674298} \frac{1}{2x^{2}},
	\end{equation*}
where we have taken the process $Y_{t}$ to be fixed at the value $y=27.9345$ for all $t$.  This was obtained by solving the Merton HJB equation 
	\begin{equation*}
	\begin{cases}
	v_{t} - \frac{1}{2} \lambda^{2}(y) \frac{v_{x}^{2}}{v_{xx}} = 0 \\
	v(T,x) = -\frac{1}{2x^{2}},
	\end{cases}
	\end{equation*}
with $\lambda^{2}(y) = 0.0002354511446$ at $y = 27.9345$.  The corresponding Merton trading strategy is then given by
	\begin{equation*}
	\pi^{\text{Mer}}(x) = -\frac{\lambda(y)\Umer_{x}}{\sigma(y)\Umer_{xx}},
	\end{equation*}
which simplifies to 
	\begin{equation} \label{MertonPortfolio}
	 \pi^{\text{Mer}}(x) = 0.755162649999999x.
	 \end{equation}

We now compare, at time $t = 0$, our approximation to the optimal portfolio, given by \eqref{finitehorizonpi}, against the actual optimal portfolio $\pi_{U}(x) \approx 0.745029x$ (obtained by substituting \eqref{exvalueformula} into \eqref{maxport} and evaluating at $t=0$), and the Merton portfolio \eqref{MertonPortfolio} (see figures \ref{fig:schemeportsunrestricted} and \ref{fig:schemeportsrestricted}).  We also graph the value function \eqref{exvalueformula} against our approximation to the value function given by \eqref{valueapproximationscheme} (see figures \ref{fig:schemewealthunrestricted} and \ref{fig:schemewealthrestricted}). \setcounter{figure}{0}
	\begin{figure}[H]
		\begin{minipage}[t]{.48\linewidth}
		\begin{center} \includegraphics[scale=0.43]{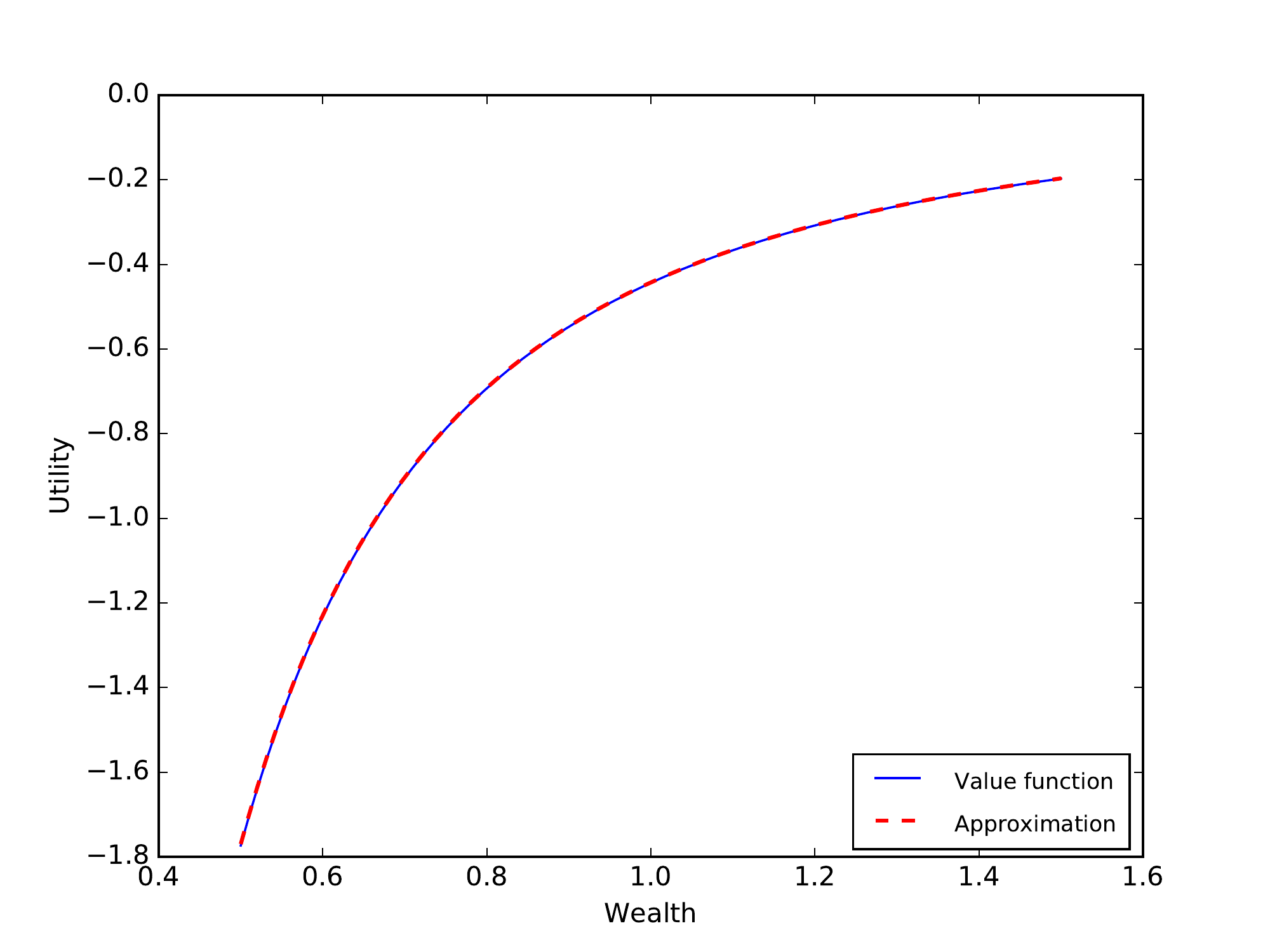} \end{center}
		\caption{($t = 0, \,T=2, \,n=4$) The value function plotted against the approximation obtained via the scheme described by \eqref{valueapproximationscheme}.}
		\label{fig:schemewealthunrestricted}
		\end{minipage}
\hfill
 		\begin{minipage}[t]{.48\linewidth}
		\begin{center} \includegraphics[scale=0.43]{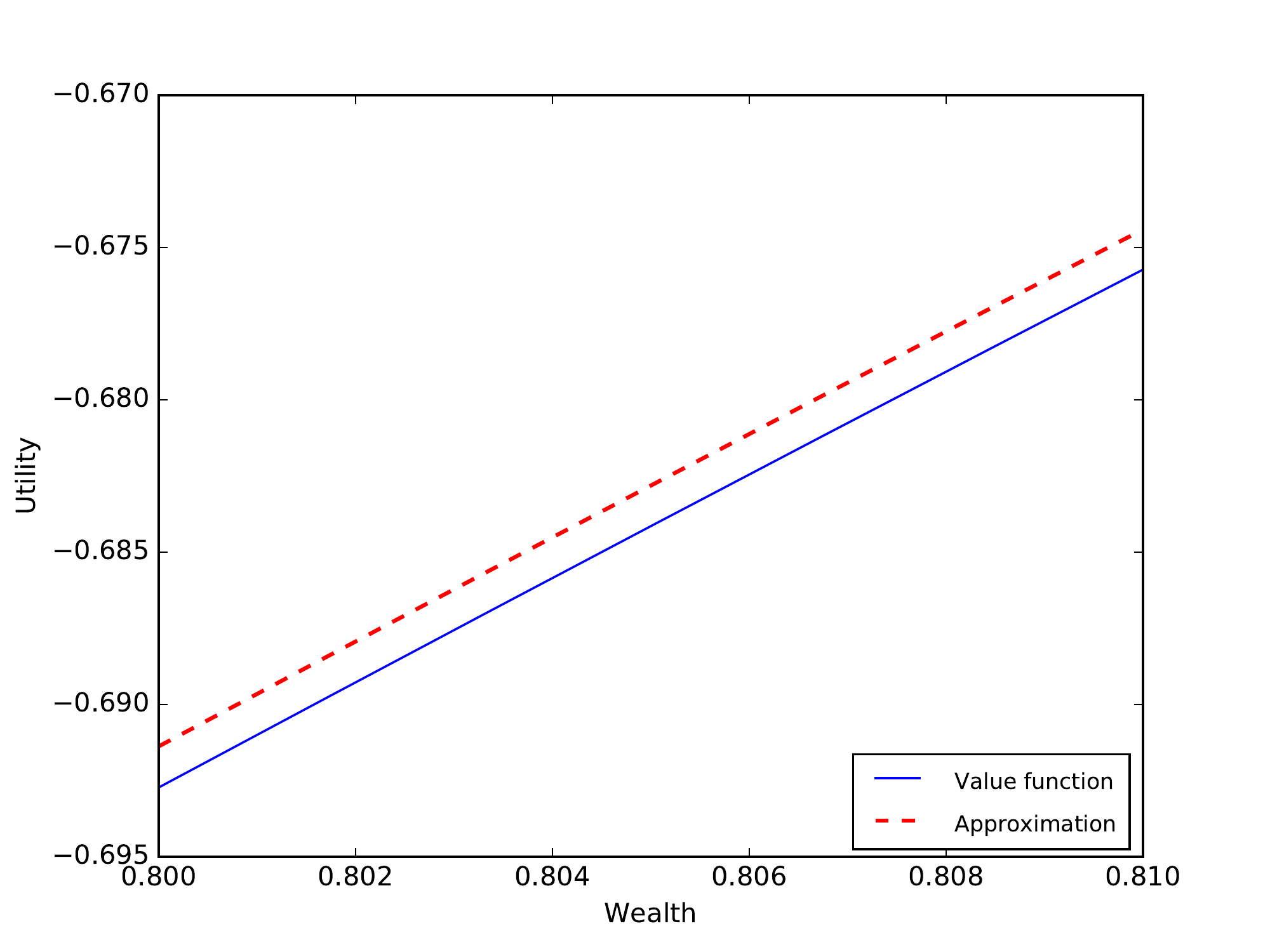} \end{center}
		\caption{($t=0, \,T=2, \,n=4$) When the wealth interval of figure \ref{fig:schemewealthunrestricted} is shortened. }
		\label{fig:schemewealthrestricted}
		\end{minipage}
	\end{figure}
	\begin{figure}[H]
		\begin{minipage}[t]{.48\textwidth}
		\begin{center} \includegraphics[scale=0.43]{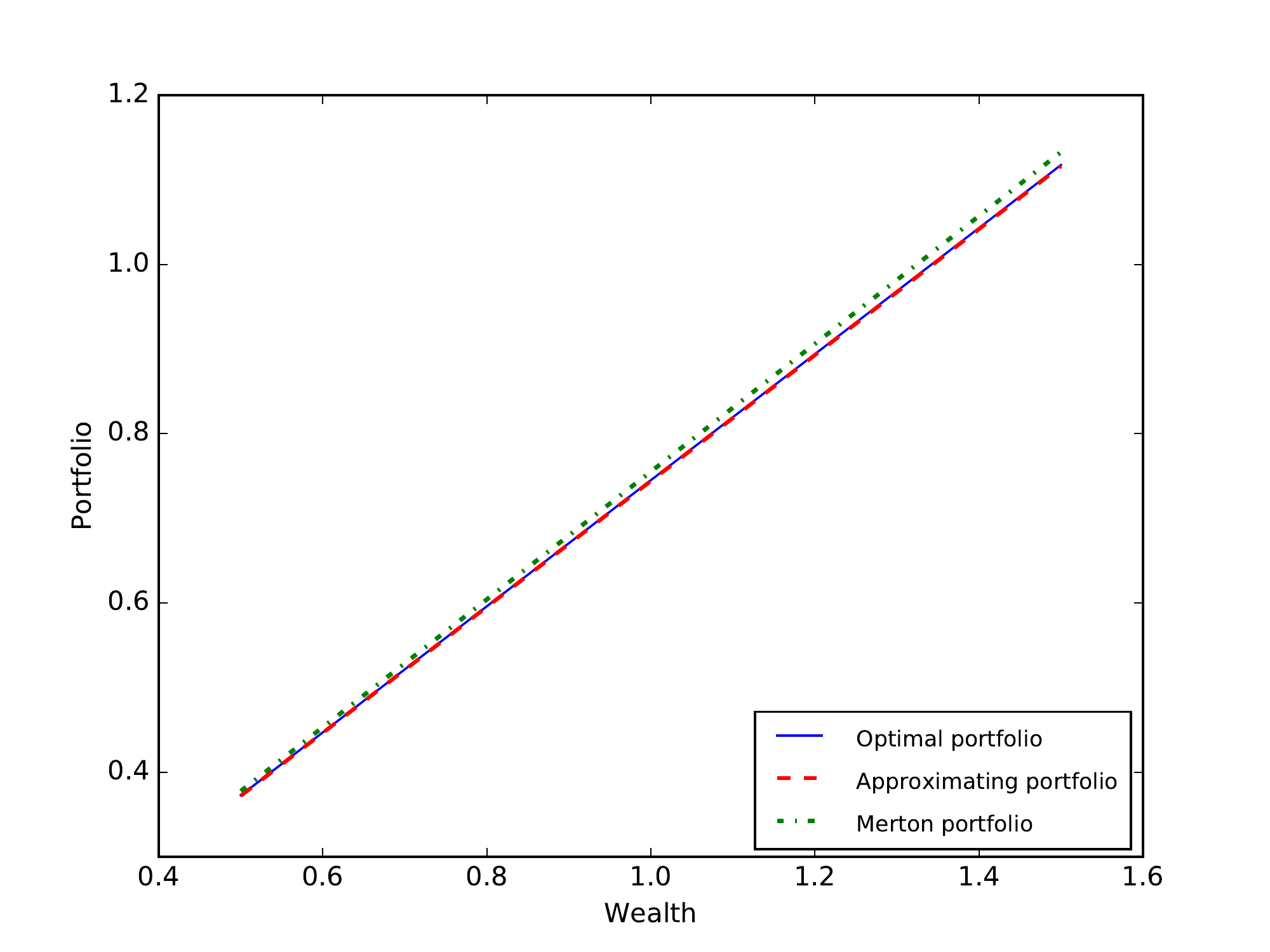} \end{center}
		\caption{($t=0, \,T=2, \,n=4$) The optimal portfolio plotted against the approximating portfolio generated by scheme \eqref{valueapproximationscheme}, and the Merton portfolio given in \eqref{MertonPortfolio}.}
		\label{fig:schemeportsunrestricted}
		\end{minipage}
\hfill
		\begin{minipage}[t]{.48\textwidth}
		\begin{center} \includegraphics[scale=0.43]{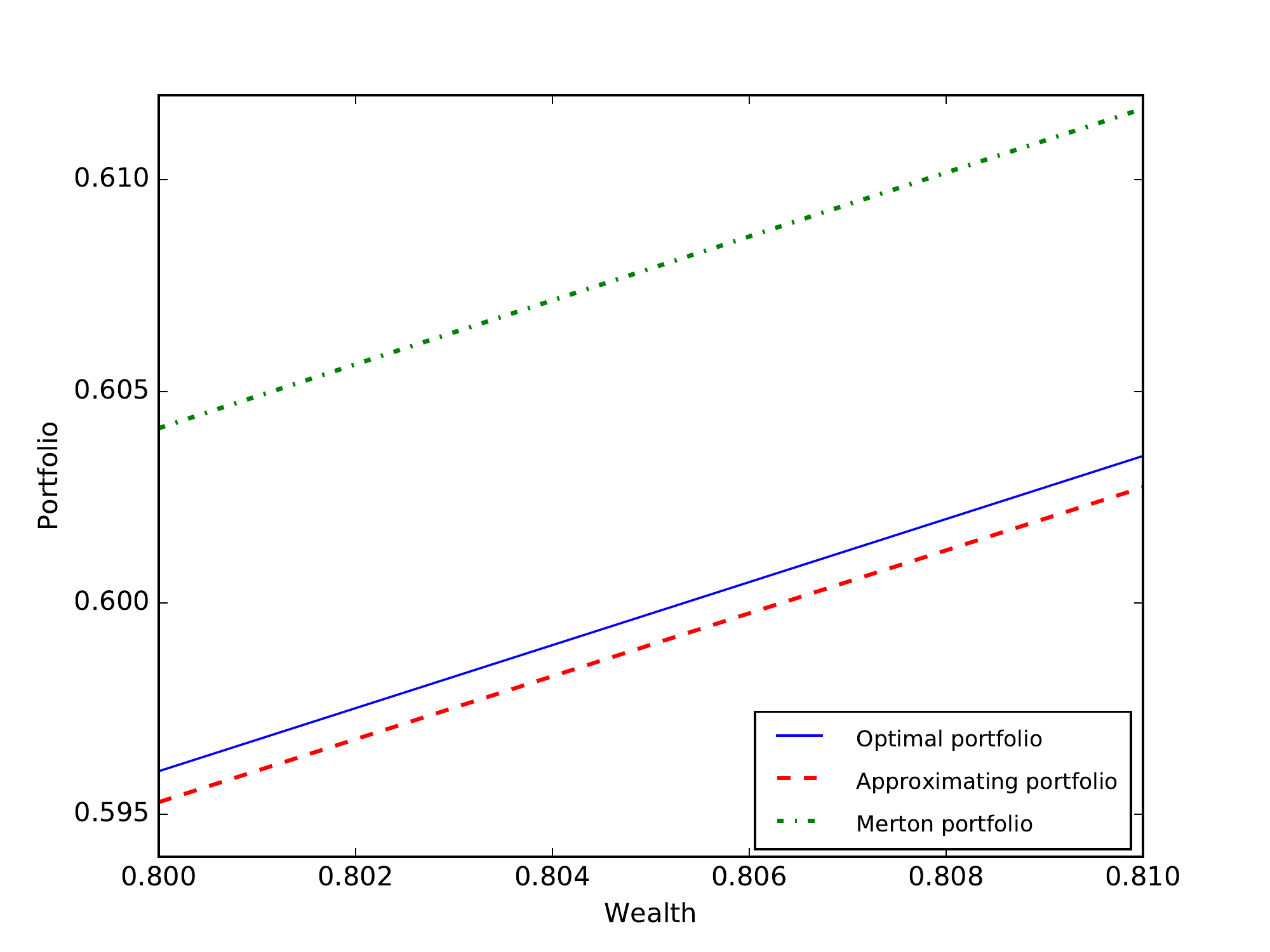} \end{center}
		\caption{($t=0, \,T=2, \,n=4$) When the wealth interval of figure \ref{fig:schemeportsunrestricted} is shortened, the accuracy of the approximation is more apparent.}
		\label{fig:schemeportsrestricted}
		\end{minipage}
	\end{figure}

\section{Appendix}

\subsection{Admissibility of $\pih$}
	\begin{lemma} \label{admissibilitypihat}
Let $X^{\pih}_s$ be the wealth process given by the SDE \eqref{dscw} under portfolio 
$\pih_{t} := \pih(t,X^{\pih}_{t},Y_{t})$ (with $\pih(t,x,y)$ as defined in \eqref{approxoptimalport}) and assume $X^{\pih}_{t} = x \in (0,\infty)$.  Under Assumptions \ref{modelassumptions} and \ref{utilityassumptions}, the strategy $\pih_{t}$ is admissible as defined in Definition \ref{admissibledefinition}.
	\end{lemma}
	
	\begin{proof}
	We begin by noting that progressive measurability of $\pih_{t}$ follows from the continuity of this function in the variable $t$.  In addition, Assumptions \ref{modelassumptions} and \ref{utilityassumptions} and the definition of $\pih(t,x,y)$ in \eqref{approxoptimalport} imply 
		\begin{equation} \label{linearboundpihat}
		\left | \sigma(y)\frac{\pih(t,x,y)}{x} \right| \leq C_{1}
		\end{equation} 
for some constant $C_{1}$, which, along with \eqref{dscw}, immediately implies that $\pih_{t}$ yields a strictly positive wealth process.  It remains to be shown that $\pih_{t}$ satisfies $E \left [\displaystyle \int \limits_{0}^{T} \sigma_{s}^{2} \pih_{s}^{2} \,ds \right ]  + E \left [\displaystyle \int \limits_{0}^{T} (X^{\pih}_{s})^{-2\gamma} \sigma_{s}^{2}\pih_{s}^{2}  \,ds \right ]  < \infty$, with $\gamma$ as in Definition \ref{admissibledefinition}.  Note that by \eqref{linearboundpihat}, we have $|\sigma_{s}^{2} \pih_{s}^{2}| \leq c_{1}(X^{\pih}_{s})^{2}$ for some constant $c_{1}$, so it is enough to show that $(X^{\pih}_{\cdot})^{p}$ is integrable, where $p= 2$ or $p = -2\gamma$.

	We apply Ito's formula to $\log (X^{\pih}_{\cdot})^{p}$ to get
	\begin{align*}
	\log (X^{\pih}_u)^{p}&= p\log X^{\pih}_u\\
&=\log x^{p} + p\int_t^u\left[\sigma(Y_s)\lambda(Y_s)\frac{\pih_s}{X^{\pih}_s}-\frac{1}{2}\sigma^2(Y_s)\left(\frac{\pih_s}{X^{\pih}_s}\right)^2 \right]ds + p\int_t^u\sigma(Y_s)\frac{\pih_s}{X^{\pih}_s}dW^{1}_s.
	\end{align*}
Let $M_u:=p\int_t^u\sigma(Y_s)\frac{\pih_s}{X^{\pih}_s}dW^{1}_s$, for $t\leq u\leq T$. From the boundedness of $\left|\frac{\sigma(y)\pih(t,x,y)}{x}\right|$ and $\lambda(y)$, it follows that there exists $c_{2}>0$ such that 
	\begin{align*}
	\log (X^{\pih}_u)^{p}&\leq \log x^{p}+c_{2}(u-t)+M_u\\
&\leq \log x^{p}+c_{2}T+M_u.
	\end{align*}
So \[(X^{\pih}_u)^{p}\leq x^{p}e^{c_{2}T}e^{M_u}.\]
Define $Z_u:=e^{M_u-\frac{1}{2}[M]_u}$ where $[M]_u=p^{2}\int_t^u \sigma_{s}^{2}\left(\frac{\pih_s}{X^{\pih}_s}\right)^2ds$. Since $\sigma(y)\frac{\pih(t,x,y)}{x}$ is a bounded function, $Z$ is a square-integrable martingale.  Then for a constant $c_{3} > 0$, we can write
\begin{equation} \label{uniformbound}
 (X^{\pih}_u)^{p}\leq x^{p}e^{\C3 T}Z_u\leq x^{p}e^{c_{3} T}\sup_{t\leq u\leq T}Z_u.
 \end{equation}
Define $K:=x^{p}e^{c_{3} T}\sup_{t\leq u\leq T}Z_u$, then 
\[E|K|\leq x^{p}e^{c_{3} T}(1+4EZ^2_T)<\infty.\]
This shows $\{(X^{\pih}_u)^{p}\}_{u\in[t,T]}$ is bounded uniformly in $u$ by an integrable random variable, and thus establishes 
	\begin{equation*}
	E \left [\displaystyle \int \limits_{0}^{T} \sigma_{s}^{2} \pih_{s}^{2} \,ds \right ]  + E \left [\displaystyle \int \limits_{0}^{T} (X^{\pih}_{s})^{-2\gamma} \sigma_{s}^{2}\pih_{s}^{2}  \,ds \right ]  < \infty.
	\end{equation*}
	\end{proof}

\subsection{Uniform bound of $G(X^{\pi}_{T\wedge \tau_{n}})$}
	\begin{lemma} \label{boundedintfunc}
Let $X^{\pi}_s$ be the wealth process given by the SDE \eqref{dscw} under the arbitrary admissible portfolio 
$\pi(t,X^{\pi}_{t},Y_{t})$ and assume $X^{\pi}_{t} = x$.  Under the assumptions of Theorem \ref{maintheorem},  $\{G(X^{\pi}_{T\wedge \tau_{n}})\}_{n =1}^{\infty}$ is uniformly bounded by an integrable random variable, where $G(x) = \log(x) + 1$ under Case 1 of Assumption \ref{utilityassumptions}, and $G(x) = x^{1-\alpha} + x^{1-\beta}$ for positive $\alpha, \beta \neq 1$, under Case 2 of Assumption \ref{utilityassumptions}.
	\end{lemma}
	\begin{proof}
We will first consider case 1 where $G(x) = x^{1-\alpha} + x^{1-\beta}$.  Note that it is enough to prove $\{(X^{\pi}_{T\wedge \tau_{n}})^{1-\gamma}\}_{n=1}^{\infty}$, for some positive $\gamma \neq 1$, is uniformly bounded by an integrable random variable.  

If $0 < \gamma < 1$, then Young's inequality gives $(X^{\pi}_{T\wedge \tau_{n}})^{1-\gamma} \leq C \left(1 + (X^{\pi}_{T\wedge \tau_{n}})^{2} \right)$ for some constant $C$.  If we set $M_{u} := \int \limits_{t}^{u} \sigma_{s} \pi_{s} \,dW^{1}_{s}$, then $M_{u}$ is a martingale, and \eqref{dscw} gives 
	\begin{equation*}
	(X^{\pi}_{T\wedge \tau_{n}})^{2} \leq \C1\left (x^{2} + \int \limits_{0}^{T} \sigma_{s}^{2} \pi_{s}^{2} \lambda_{s}^{2} \,dt + \left(\sup \limits_{u \in [t,T]} M_{u}\right)^{2} \right )
	\end{equation*} 
for some constant $\C1$.  In particular, taking expectation yields, by Doob's maximal inequality, 
	\begin{equation*}
	E\left [(X^{\pi}_{T\wedge \tau_{n}})^{2} \right ] \leq \C2 \left (1 + E \left [(M_{T})^{2} \right ] \right)  = \C2\left (1 + E\left [\int \limits_{0}^{T} \sigma_{s}^{2} \pi_{s}^{2} \,ds \right ] \right) < \infty,
	\end{equation*}
for some constant $\C2$.  Therefore, $\{(X^{\pi}_{T\wedge \tau_{n}})^{2}\}_{n=1}^{\infty}$ is uniformly bounded by the integrable random variable $\xi := \C1\left (1 + x^{2} + \displaystyle \int \limits_{0}^{T} \sigma_{s}^{2} \pi_{s}^{2} \lambda_{s}^{2} \,ds + \left(\sup \limits_{u \in [t,T]} M_{u}\right)^{2} \right )$, implying $(X^{\pi}_{T\wedge \tau_{n}})^{1-\gamma} \leq C \left(1 + (X^{\pi}_{T\wedge \tau_{n}})^{2} \right)$ is uniformly bounded in $n$ by an integrable random variable.

In the case of $\gamma > 1$, we apply Ito's formula to $(X^{\pi}_{u})^{1-\gamma}$ to obtain 
	\begin{equation*}
	(X^{\pi}_{u})^{1-\gamma} = x^{1-\gamma} + (1-\gamma) \int \limits_{t}^{u} \lambda_{s} \sigma_{s} \pi_{s} (X^{\pi}_{s})^{-\gamma} - \frac{(1-\gamma)\gamma}{2} \sigma_{s}^{2} \pi_{s}^{2} (X^{\pi}_{s})^{-\gamma -1}  \,ds + (1-\gamma)\int \limits_{t}^{u} \sigma_{s} \pi_{s} (X^{\pi}_{s})^{-\gamma} \,dW^{1}_{s}.
	\end{equation*}
Thus, letting $Z_{u} := \int \limits_{t}^{u} \sigma_{s} \pi_{s} (X^{\pi}_{s})^{-\gamma} \,dW^{1}_{u}$, which is a martingale by the admissibility of $\pi_{s}$ (see Definition \ref{admissibledefinition}), we see that  
	\begin{equation} \label{ItoBound} 
	(X^{\pi}_{u})^{1-\gamma} \leq \C3\left(1 + \int \limits_{0}^{T} \sigma_{s}^{2} \pi_{s}^{2}  (X^{\pi}_{s})^{-2\gamma} + \sigma_{s}^{2} \pi_{s}^{2} (X^{\pi}_{s})^{-2} \,ds + \left (\sup \limits_{s \in [0,T]} Z_{u}\right)^{2}\right),
	\end{equation}
for some constant $\C3$.  Taking the expectation of both sides and applying Doob's maximal inequality gives
	\begin{equation*}
	E[(X^{\pi}_{u})^{1-\gamma}] \leq \C4\left(1 + E\left[\int \limits_{0}^{T} \sigma_{s}^{2} \pi_{s}^{2}  (X^{\pi}_{s})^{-2\gamma} + \sigma_{s}^{2} \pi_{s}^{2} (X^{\pi}_{s})^{-2} \,ds\right] + E[Z_{T}^{2}] \right).
	\end{equation*}
As $E[Z_{T}^{2}] = E\left [ \displaystyle \int \limits_{0}^{T} \sigma_{s}^{2} \pi_{s}^{2} (X^{\pi}_{s})^{-2\gamma} \,ds \right] < \infty$, the right hand side of the above inequality is finite.  Thus, the right hand side of inequality \eqref{ItoBound} serves as the uniform bound for $(X^{\pi}_{T\wedge \tau_{n}})^{1-\gamma}$ when $\gamma > 1$.

To prove the result in the case that $G(x) = 1 + \log(x)$, it is enough to show $\{\log(X^{\pi}_{T\wedge \tau_{n}})\}_{n=}^{\infty}$ is uniformly bounded by an integrable random variable.  We apply Ito's formula to $\log (X^{\pi}_{\cdot})$ to get
	\begin{equation*}
	\log (X^{\pi}_u) =\log x + \int_t^u\left[\sigma(Y_s)\lambda(Y_s)\frac{\pi_s}{X^{\pi}_s}-\frac{1}{2}\sigma^2(Y_s)\left(\frac{\pi_s}{X^{\pi}_s}\right)^2 \right]ds + \int_t^u \sigma(Y_s)\frac{\pi_s}{X^{\pi}_s}dW^{1}_s.
	\end{equation*}
As in the power case, we let $\E_{u} := \displaystyle \int_{t}^{u} \sigma(Y_{s})\frac{\pi_{s}}{X_{s}^{\pi}}dW^{1}_s$.  Then, by the admissibility of $\pi_{s}$ in Definition \ref{admissibledefinition}, $\E_{u}$ is a martingale, and so we can apply Doob's maximal inequality as above.  In particular, we have
	\begin{equation} \label{ItoBound2}
	\log(X^{\pi}_{u}) \leq \C5\left(1 + \int \limits_{0}^{T} \sigma^{2} \pi^{2} (X^{\pi}_{s})^{-2}\,ds + \left (\sup \limits_{s \in [t,T]} \E_{s}\right )^{2}\right)
	\end{equation}
for some constant $\C5$.  Thus, taking expectations gives 
	\begin{equation*}
	E[\log(X^{\pi}_{u})] \leq \C5\left (1 + E \left [\displaystyle \int \limits_{0}^{T} \sigma^{2} \pi^{2} (X^{\pi}_{s})^{-2} \,ds \right ] + E\left [\left (\sup \limits_{s \in [t,T]} \E_{s}\right )^{2} \right] \right ),
	\end{equation*}
and Doob's inequality gives, for some constant $\C6$, 
	\begin{equation*}
	E \left [\left (\sup \limits_{s \in [t,T]} \E_{s}\right)^{2} \right ] \leq \C6 E[\E_{T}^{2}] = \C6E\left [ \displaystyle \int \limits_{t}^{T} \sigma^{2} \pi^{2} (X^{\pi}_{s})^{-2} \,ds \right] < \infty,
	\end{equation*} 
with the last inequality following by the definition of admissibility given in Definition \ref{admissibledefinition}.  This establishes the right hand side of \eqref{ItoBound2} as the integrable random variable which uniformly bounds $\{\log\left(X^{\pi}_{T \wedge \tau_{n}}\right)\}_{n = 1}^{\infty}$.
\end{proof}

\end{document}